\documentclass[preprint,10pt]{article}

\usepackage{amsfonts}
\usepackage{amsmath,amsthm}
\usepackage{amssymb}
\usepackage{enumerate}
\usepackage[english]{babel}
\usepackage[utf8]{inputenc}
\usepackage{bbm}
\usepackage{mathrsfs}
\usepackage{color}
\usepackage{multirow}
\usepackage{graphicx}
\usepackage{epsfig}
\usepackage[colorlinks=true,citecolor=red,linkcolor=blue,pdfpagetransition=Blinds]{hyperref}
\usepackage{fullpage}
\usepackage{pgfplotstable}
\usepackage{lipsum}
\usepackage{cases}
\usepackage{tikz}
\providecommand{\keywords}[1]{\noindent {\textbf{Keywords:}} #1}

\usepackage{color}

\newtheorem{theorem}{Theorem}

\newtheorem{definition}[theorem]{Definition}

\newtheorem{lemma}[theorem]{Lemma}

\newtheorem{proposition}[theorem]{Proposition}
\newtheorem{remark}[theorem]{Remark}

\numberwithin{equation}{section} 
\numberwithin{theorem}{section}
\providecommand{\AMScodes}[1]{\noindent {\textbf{MSC2010:}} #1}

\def\dx{\,\textnormal{d}x}
\def\dt{\textnormal{d}t}
\def\d{\,\textnormal{d}}
\def\cbd{\Gamma}

\def\supp{\textnormal{supp}\,}
\def\csbd{\rho_{\Gamma}}
\newcommand\csin[1]{\chi_{#1}}
\def\dx{\,\textnormal{d}x}
\def\dt{\textnormal{d}t}
\def\d{\,\textnormal{d}}

\begin{document}

\title{\bf Some remarks on the Robust Stackelberg controllability for the heat equation with controls on the boundary}

\author{ V\'ictor Hern\'andez-Santamar\'ia \thanks{Institut de Math\'{e}matiques de
    Toulouse, UMR 5219,
    Universit\'e de Toulouse, CNRS, UPS IMT, F-31062 Toulouse Cedex 9,
    France. E-mail: \texttt{victor.santamaria@math.univ-toulouse.fr}} \and Liliana Peralta \thanks{Centro de Investigaci\'on en Matem\'aticas, UAEH, Carretera Pachuca-Tulancingo km 4.5 Pachuca, Hidalgo 42184, Mexico  E-mail: \texttt{liliana\_peralta@uaeh.edu.mx}}}

\maketitle

\abstract{
In this paper, we present some controllability results for the heat equation in the framework of hierarchic control. We present a Stackelberg strategy combining the concept of controllability with robustness: the main control (the leader) is in charge of a null-controllability objective while a secondary control (the follower) solves a robust control problem, this is, we look for an optimal control in the presence of the worst disturbance. We improve previous results by considering that either the leader or follower control acts on a small part of the boundary.  We also present a discussion about the possibility and limitations of placing all the involved controls on the boundary. 



\AMScodes{49J20; 93B05; 49K35}

\section{Introduction}\label{sec_intro}
%
%
Optimization and control problems arise in many applications of engineering and mathematics. Traditionally, such problems deal with a single objective: minimize cost, maximize benefit, etc., (see, e.g. \cite{geering,trelat} and the references therein). However, when studying more realistic and complex situations, it is desirable to include several different objectives and therefore the introduction of multi-objective optimization is essential. 


In the framework of control of PDEs, the so-called hierarchic control was introduced in \cite{LionsHier,LionsSta} by J.-L. Lions to study a bi-objective control problem for the wave and heat equation, respectively. In these works, the hierarchic control method is proposed as a tool to combine the concepts of optimal control and controllability. Such methodology employs the notion of Stackelberg optimization (\cite{Stackelber}) to deal with a multi-objective decision problem where one of the participants, the \emph{leader}, is in charge of a controllability objective and the other participant, the \emph{follower}, deals with an optimal control one. 

In the recent past, several authors have applied successfully the hierarchic control method for a wide variety of equations and solving different kind of objectives, see, among others, \cite{a_araujo,araruna,araruna1,carreno,Guillen,vhs_deT_rob,jesus,montoya}. In particular, in \cite{vhs_deT_rob}, the authors proposed to combine the notion of hierarchic control for the heat equation introduced in \cite{LionsSta} with the concept of robust control (see e.g. \cite{aziz,temam,temam_nonlinear}), this means,  the leader is in charge of a controllability problem while the follower aims to achieve its minimum in the presence of the worst disturbance allowed. By doing this, it is possible to design controls that work even in the presence of certain perturbations and allows to achieve system robustness. 

However, all of the previous works have one thing in common: they deal with hierarchic strategies with controls that are localized in the interior of the domain, namely, for distributed controls. As far as we know, there is only one paper dealing with the boundary case: in \cite{da_silva}, the authors study a Stackelberg-Nash strategy for  (semilinear) parabolic equations with the possibility of the leader or the followers being placed on the boundary. 

Here, employing some arguments in \cite{da_silva}, we extend and discuss the results concerning the robust hierarchic strategy for the heat equation introduced in \cite{vhs_deT_rob} using boundary controls instead of distributed ones. This change, together with the intricacy of the hierachic methodology introduce additional difficulties and new control strategies are required. 

\subsection{The problem and its formulation}

In this paper, we are interested in a robust control strategy for the heat equation where we assume that we can act on the dynamics of the system through a hierarchy of controls. We will study the case where some of the controls act through a (small) portion of the boundary. To fix ideas, we begin by explaining one of the control problems addressed in this paper. 

Let us consider a bounded open set $\Omega\subset \mathbb{R}^N$, $N\geq 1$ with boundary $\partial \Omega$ of class $\mathcal C^2$. Let $\omega\subset \Omega$ be a nonempty open subset and {$\cbd$ be a nonempty open subset of $\partial \Omega$}.  Given $T>0$, we will use the notation $Q:=\Omega\times(0,T)$ and $\Sigma:=\partial \Omega\times(0,T)$, while $n(x)$ will denote the outward unit normal vector at the point $x\in \partial \Omega$. 

Let us consider the system
\begin{equation}\label{heat_lin}
\begin{cases}
y_t-\Delta y=\csin{\omega}h+\psi, & \text{in Q}, \\
y=v\csbd &\text{on } \Sigma, \\
y(x,0)=y^0(x), & \text{in } \Omega.
\end{cases}
\end{equation}
where $y^0$ is a given initial datum and $\psi$ is an unknown perturbation.

In \eqref{heat_lin}, $y=y(x,t)$ is the state while $h=h(x,t)$ and $v=v(x,t)$ are control functions applied on $\omega$ and $\Gamma$, respectively. Here $\csin{\omega}$ is the characteristic function of the set $\omega$ and $\csbd$ is a smooth nonnegative function such that $\supp\csbd=\overline \cbd$. For the moment, we assume that the system is well defined in terms of the controls $h$ and $v$, the perturbation $\psi$, and the initial condition $y_0$. This will be clarified below. 

The intuitive idea of the robust hierarchic control is to choose ``simultaneously'' the control functions  $v$ and $h$ in such way that the following optimality problems are solved: 
\begin{enumerate}
\item find the ``best'' control $v$ such that the solution to $y$ is ``not too far'' from a desired target $y_d$ even in the presence of the ``worst'' disturbance $\psi$, and
\item find the minimal norm control $h$ such that $y(\cdot,T)=0$. 
\end{enumerate}

Seen independently, Problem 1 (i.e. $h\equiv 0$) is a classical robust control problem (cf. \cite{temam,temam_nonlinear,aziz}) which looks for a control such that a given cost functional achieves its minimum in presence of the worst disturbance possible. Problem 2 (i.e. $v\equiv\psi\equiv 0$) is a classical null controllability problem and it has been studied for a broad class of systems described by parabolic PDEs, see for instance \cite{cara_guerrero} and the references within.

Consider a nonempty open set $\mathcal O_d\subset \Omega$ and define the cost functional
\begin{equation}\label{func_rob}
J_r(v,\psi;h)=\frac{1}{2}\iint_{\mathcal O_d\times(0,T)}|y-y_d|^2\dx\dt+\frac{1}{2}\left[\ell^2\iint_{\Sigma}|v|^2\d\sigma\dt-\gamma^2\iint_Q|\psi|^2\dx\dt\right]
\end{equation}
%
%
where $\ell,\gamma>0$ are constants and $y_d\in L^2(\mathcal O_d\times(0,T))$ is given. This functional is used to formulate  Problem 1, more precisely, we will look for a saddle point $(\bar v,\bar \psi)$ which simultaneously maximize $J_r$ with respect to $\psi$ and minimize it with respect to $v$. The parameters $\ell$ and $\gamma$ play a key role and take into account the relative weight of each term: the term $-\gamma^2\|\psi\|^2_{L^2(Q)}$ constrains the magnitude of the disturbance allowed in the optimization process while the term associated to $\ell^2\|v\|^2_{L^2(\Sigma)}$ moderates the effort made by the control. 

Now, we are in position to describe the Robust hierarchic control problem. According to the formulation introduced by Lions \cite{LionsSta}, we denote $h$ as the leader control and $v$ as the follower control. Then, the proposed methodology consists of two parts:

\begin{enumerate}
\item[(i)] For a fixed leader $h\in L^2(\omega\times(0,T))$, we look for an optimal pair $(\bar v,\bar \psi)$ solving the robust control problem:
\begin{definition}\label{defi_rob}
Let $h\in L^2(\omega\times(0,T))$ be fixed. The control $\bar v\in L^2(\Sigma)$, the disturbance $\bar \psi\in L^2(Q)$ and the state $\bar y=\bar y(h,\bar v,\bar \psi)$ solution to \eqref{heat_lin} associated to $(\bar v,\bar \psi)$ are said to solve the robust control problem when a saddle point $(\bar v,\bar \psi)$ (which depends on $h$) of the cost functional $J_r$ is reached, namely
\begin{equation}\label{saddle}
J_r(\bar v,\psi;h)\leq J_r(\bar v,\bar \psi;h)\leq J_r(v,\bar \psi;h), \quad \forall (v,\psi)\in L^2(\Sigma)\times L^2(Q).
\end{equation}
In this case,
\begin{equation}\label{saddle_minmax}
J_r(\bar v,\bar \psi;h)=\max_{\psi\in L^2(Q)}\min_{v\in L^2(\Sigma)}J_r(v,\psi;h)=\min_{v\in L^2(\Sigma)}\max_{\psi\in L^2(Q)}J_r(v,\psi;h).
\end{equation}
\end{definition} 
Under certain conditions, we will prove that there exists a unique pair $(\bar v,\bar \psi)$ and the associated state $\bar y=\bar y(h,\bar v,\bar\psi)$ such that \eqref{saddle} holds.
\item[(ii)] After identifying the saddle point for each $h$, we look for the control of minimal norm $\bar h$ satisfying null controllability constraints, i.e., we look for an optimal control $\bar h$  such that
\begin{equation}\label{opt_leader}
J(\bar h)=\min_{h\in L^2(\omega\times(0,T))}\frac{1}{2}\iint_{\omega\times(0,T)}|h|^2\dx\dt, \quad \text{subject to } y(\cdot,T;\bar v,\bar \psi)=0.
\end{equation}
\end{enumerate}

\begin{remark}
As in \cite{LionsSta} and other related papers, we address the multi-objective optimization problem by solving the mono-objective problems \eqref{saddle_minmax} and \eqref{opt_leader}. Note, however, that in the second minimization problem the solution of the robust control $(\bar v, \bar \psi)$ is fixed and therefore its characterization needs to be considered.
\end{remark}

\subsection{Main results}

In a first result, we address the robust hierarchic control of system \eqref{heat_lin}, that is, the case where the follower control is applied on the boundary and the leader control is localized on the interior of the domain. In this regard, we have the following result.
\begin{theorem}\label{teo_main1} 
Assume that $\omega\cap\mathcal O_d\neq \emptyset$. Then, there exist $\gamma_0$, $\ell_0$ and a positive function $\varrho_1=\varrho_1(t)$ blowing up at $t=T$ such that for any $\gamma>\gamma_0$, $\ell>\ell_0$, $y^0\in H^{-1}(\Omega)$ and $y_d\in L^2(\mathcal O_d\times(0,T))$ verifying 
\begin{equation}\label{integ_yd}
\iint_{\mathcal O_d\times(0,T)}\varrho_1^2|y_d|^2dxdt<+\infty,
\end{equation}
we can find a leader control $h\in L^2(\omega\times(0,T))$ and a unique solution $(\bar v,\bar\psi)\in L^2(\Sigma)\times L^2(Q)$ to the robust control problem for the cost functional \eqref{func_rob}, such that the associated solution of \eqref{heat_lin} verifies $y(\cdot,T)=0$ in $\Omega$. 
\end{theorem}

As usual in robust control problems, the assumption on $\gamma$ means that the possible disturbances {spoiling the control objectives should have moderate $L^2$-norms.} Indeed, without this condition, it is not possible to prove the existence of the saddle point. On the other hand, condition \eqref{integ_yd} on the target $y_d$ means that it should approach 0 as $t\to T$. This is a standard feature in some null controllability problems and has been discussed, for instance, in \cite{araruna,deteresa2000}. 

We shall prove Theorem \ref{teo_main1} in two steps. In the first one, we will adapt the methodology in \cite{vhs_deT_rob} to the boundary case  to prove the existence and uniqueness of a saddle point to \eqref{func_rob}. Then, the solution can be characterized by means of an optimality system leading to a coupled system. In the second part, we will use Carleman estimates for parabolic equations with non-homogenous boundary terms to deduce an observability inequality for a suitable adjoint system, which will imply the desired null controllability objective. 

In this paper, we are also interested in studying different configurations for the positioning of the boundary controls. A first question that arises naturally is the possibility to exchange the position of the leader $h$ and the follower control $v$. Adapting some of the arguments in \cite{da_silva}, we will see that this is in fact possible by considering systems of the form 
\begin{equation}\label{heat_dif}
\begin{cases}
z_t-\Delta z=\csin{\mathcal B_1}v+\csin{\mathcal B_2}\psi, & \text{in Q}, \\
z=h\csin{\Gamma} &\text{on } \Sigma, \\
z(x,0)=z^0(x), & \text{in } \Omega,
\end{cases}
\end{equation}
where {$\mathcal B_i \subsetneq \Omega$}, $i=1,2,$ are nonempty open subsets and ${z^0(x)\in H^{-1}(\Omega)}$ is a given initial datum. 

The same methodology presented above can be used to address the robust hierarchic control of \eqref{heat_dif}. In this case, the cost functional \eqref{func_rob} should be replaced by 
\begin{equation}\label{func_rob_dif}
K_r(v,\psi;h)=\frac{1}{2}\iint_{\mathcal O_d\times(0,T)}|y-y_d|^2\dx\dt+\frac{1}{2}\left[\ell^2\iint_{\mathcal B_1\times(0,T)}|v|^2\dx\dt-\gamma^2\iint_{\mathcal B_2\times(0,T)}|\psi|^2\dx\dt\right].
\end{equation}
As before, we will see that the robust control can be solved by selecting appropriate parameters $\ell$ and $\gamma$. Notice that unlike \eqref{heat_dif}, here we maximize for disturbances $\psi$ localized in the region $\mathcal B_2$. This comes from a technical reason concerning the resolution of the null controllability objective (see Section {\ref{sec_bound_leader}} for details), but which is not necessary to solve the robust control part. 

Once a characterization for the saddle point of $K_r$ is known, the resulting optimality system is once again a coupled system of PDEs. It is well-known that controllability problems using boundary controls is a difficult task for systems of two or more equations (see, e.g, \cite{assia_survey,assia_luz_new}). Here,  using that the parameters $\ell,\gamma$ coming from the solution of the robust control part are sufficiently large, we will combine Carleman estimates with boundary observations and weighted energy estimates to obtain an observability inequality for a system of two equations with only one observation at the boundary.

The result can be summarized as follows. 
\begin{theorem}\label{teo_main2}
Assume that 
\begin{equation}\label{loc_teo2}
\overline{\mathcal O_d}\cap \overline{\mathcal B_i} =\emptyset \quad\text{and}\quad \overline{\Gamma}\subset \partial \mathcal B_i, \quad i=1,2.
\end{equation}
Then, there exist $\gamma_0$, $\ell_0$ and a positive function $\varrho_2=\varrho_2(t)$ blowing up at $t=T$ such that for any $\gamma>\gamma_0$, $\ell>\ell_0$, $y^0\in H^{-1}(\Omega)$ and $y_d\in L^2(\mathcal O_d\times(0,T))$ verifying 
\begin{equation*}
\iint_{\mathcal O_d\times(0,T)}\varrho_2^2|y_d|^2dxdt<+\infty,
\end{equation*}
we can find a leader control $h\in L^2(\Gamma\times(0,T))$ and a unique solution $(\bar v,\bar\psi)\in L^2(\mathcal B_1\times(0,T))\times L^2(\mathcal B_2\times(0,T))$ to the robust control problem for \eqref{func_rob_dif}, such that the associated solution to \eqref{heat_dif} verifies $z(\cdot,T)=0$ in $\Omega$. 
\end{theorem}

Hypothesis \eqref{loc_teo2} plays a fundamental role in the selection of the weight functions participating in the Carleman estimates needed to prove Theorem \ref{teo_main2}. Indeed, it is not a common feature and allows to build two different Carleman weights which, together with the special structure of the adjoint system (see eq. \ref{adjunto1}), help to eliminate one of the boundary observation terms. This particular selection has been recently used in other hierarchic control problems, see \cite{da_silva}. 

A second question that arises in this context is if it is possible to put both controls on the boundary of the system, namely
\begin{equation}\label{sys_teo3}
\begin{cases}
w_t-\Delta w=\psi, & \text{in Q}, \\
w= h\csin{\Gamma_1}+ v{\csbd}_{2} &\text{on } \Sigma, \\
w(x,0)=w^0(x), & \text{in } \Omega.
\end{cases}
\end{equation}
where $\Gamma_{i}\subset\partial \Omega$, $i=1,2,$ are open sets with $\cbd_1\cap\cbd_2\equiv \emptyset$. We provide a partial answer to this problem for the case when $\psi\equiv 0$ and the cost functional associated to the optimization problem is defined by
\begin{equation}\label{func_teo3}
I(v;h)=\frac{1}{2}\iint_{\mathcal O_d\times(0,T)}|w-w_d|^2\dx\dt+\frac{\ell^2}{2}\iint_{\partial\Omega\times(0,T)}\rho_\star^2|v|^2\d\sigma\dt,
\end{equation}
where $\rho_\star=\rho_\star(t)$ is a suitable positive weight function blowing up exponentially at $t=0$ and $t=T$. Observe that this functional is related to a more classical optimal control problem (cf. \cite{Lions_optim,trol}) and hence we prove a result in the original framework of hierarchic control introduced in \cite{LionsSta}. We have the following theorem.
\begin{theorem}\label{teo3}
Suppose that $\psi\equiv 0$. Then, there exist $\ell_0$ and a positive function $\varrho_3=\varrho_3(t)$ blowing up at $t=T$ such that for any $\ell>\ell_0$, $w_0\in H^{-1}(\Omega)$ and $y_d\in L^2(\mathcal O_d\times(0,T))$ such that
\begin{equation*}
\iint_{\mathcal O_d\times(0,T)}\varrho_3^2|y_d|^2dxdt<+\infty,
\end{equation*}
we can find a leader control $h\in L^2(\cbd\times(0,T))$ and a unique follower control $\bar v$ minimizing \eqref{func_teo3}, such that the associated solution to \eqref{sys_teo3} verifies $w(\cdot,T)=0$.
\end{theorem}

The role of the weight function  $\rho_\star$ is clear. By minimizing \eqref{func_teo3}, we enforce the follower $v$ to vanish at $t=0$ and $t=T$ and therefore the leader $h$ finds no obstruction to control the system during the second part of the hierarchic methodology. We refer to \cite{araruna} and \cite{vhs_corri} for a similar use of weighted functionals as \eqref{func_teo3}. 

The rest of the paper is organized as follows. In section \ref{bound_follow}, we study the robust hierarchic problem concerning the case of boundary follower and distributed leader, this is, the case of system \eqref{heat_lin}. On the first stage, we address the robust control problem with boundary control and then we deal with the null controllability of the resulting system. These allow to establish the proof of Theorem \ref{teo_main1}. In section \ref{sec_bound_leader}, we deal with the case of a boundary leader. We will make special emphasis on how hypothesis \eqref{loc_teo2} helps to obtain an observability inequality for a coupled system with only one observation at the boundary. Section \ref{sec_bound} is devoted to analyze a hierarchic control strategy (in the sense of \cite{LionsSta}) for system \eqref{func_teo3} in the case when $\psi\equiv 0$. Due to the special selection of the cost functional \eqref{func_teo3}, we are able to proof Theorem \ref{teo3}. Finally, in Section \ref{sec_conclusion} we make some concluding remarks. 

\section{The case with boundary follower and distributed leader}\label{bound_follow}

\subsection{Existence, uniqueness and characterization of the saddle point}\label{ex_uniq_saddle}


In the first step, we analyze the robust control problem associated to \eqref{func_rob} and we establish conditions on the parameters $\ell$ and $\gamma$ leading to the existence of a saddle point. The proof is closely related to the one presented in \cite{vhs_deT_rob} but, for completeness, we sketch it briefly. In what follows, we assume that the leader has made a choice $h$. 

The first thing to check is that the solution $y$ to \eqref{heat_lin} is well posed and is uniquely determined by the data of the problem. It is well-known (see, e.g., \cite{lions_magenes}) that for any $v\in L^2(\Sigma)$, $\psi\in L^2(Q)$, $h\in L^2(\omega\times(0,T))$ and $y_0\in H^{-1}(\Omega)$, system \eqref{heat_lin} admits a unique weak solution (defined by transposition) that satisfies
\begin{equation}\label{space_trans}
y\in L^2(Q)\cap C^0([0,T];H^{-1}(\Omega)).
\end{equation} 

Moreover, $y$ satisfies an estimate of the form
\begin{equation}\label{ener_trans}
\|y\|_{L^2(Q)}\leq C\left(\|y_0\|_{H^{-1}(\Omega)}+\|v\|_{L^2(\Sigma)}+\|h\|_{L^2(\omega\times(0,T))}+\|\psi\|_{L^2(Q)}\right)
\end{equation}
where $C$ is a positive constant not depending on $\psi$, $v$, $h$ nor $y_0$. 

\begin{remark}
Note that the regularity \eqref{space_trans} and the energy estimate \eqref{ener_trans} hold also for the boundary control systems \eqref{heat_dif} and \eqref{sys_teo3} by making suitable changes. 
\end{remark}

The main goal of this section is to proof the existence and uniqueness of a saddle point $(\bar v,\bar\psi)$ to the robust control problem in Definition \ref{defi_rob}. The result is based on the following result. 
\begin{proposition}\label{prop_saddle}
Let $J$ be a functional defined on $X\times Y$, where $X$ and $Y$ are non-empty, closed, unbounded, convex sets. If $J$ satisfies
\begin{enumerate}
\item $\forall \psi\in Y$, $v\mapsto J(v,\psi)$ is convex lower semicontinuous,
\item $\forall v\in X$, $\psi\mapsto J(v,\psi)$ is concave upper semicontinuous,
\item $\exists \psi_0\in X$ such that $\lim_{\|v\|_{X}\to+\infty}J(v,\psi_0)=+\infty$, 
\item $\exists v_0\in Y$ such that $\lim_{\|\psi\|_Y\to+\infty}J(v_0,\psi)=-\infty$,
\end{enumerate}
then the functional $J$ has at least one saddle point $(\bar v,\bar \psi)$ and
\begin{equation*}
J(\bar v,\bar \psi)=\min_{v\in X}\sup_{\psi\in Y} J(v,\psi)=\max_{\psi\in Y}\min_{v\in X}J(v,\psi).
\end{equation*}
Moreover, if $\{\psi\mapsto J(v,\psi)\}$ is strictly concave $\forall v\in X$ and $\{v\mapsto J(v,\psi)\}$ is strictly convex $\forall \psi\in Y$, the saddle point $(\bar v,\bar \psi)$ is unique. 
\end{proposition}

 The proof can be found on \cite[Prop. 1.5 and 2.2, Ch. VI]{Ekeland}. The aim here is to apply Proposition \ref{prop_saddle} to functional \eqref{func_rob} with $X=L^2(\Sigma)$ and $Y=L^2(Q)$. To verify conditions 1--4 for our problem, we need the following auxiliary lemma.

\begin{lemma}\label{lemma_prop_sol1}
Let $h\in L^2(\omega\times(0,T))$ and $y_0\in H^{-1}(\Omega)$ be given. The mapping $(v,\psi)\mapsto y(v,\psi)$ from $L^2(\Sigma)\times L^2(Q)$ into $L^2(Q)$ is affine, continuous, and has G\^{a}teaux derivative $y^\prime(v^\prime,\psi^\prime)$ in every direction $(v^\prime,\psi^\prime)\in L^2(\Sigma)\times L^2(Q)$. Moreover, the derivative $y^\prime(v^\prime,\psi^\prime)$ solves the system
\begin{equation}\label{deriv_sys1}
\begin{cases}
y^\prime_t-\Delta y^\prime=\psi^\prime \quad &\textnormal{in $Q$}, \\
y^\prime=v^\prime\csbd \quad &\textnormal{on $\Sigma$,} \\
y^\prime(x,0)=0 \quad &\textnormal{in $\Omega$},
\end{cases}
\end{equation}
\end{lemma}

\begin{proof}
The fact that the mapping $(v,\psi)\mapsto y(v,\psi)$ is affine and continuous follow from the linearity of system \eqref{heat_lin} and the energy estimate \eqref{ener_trans}. The existence of the G\^{a}teaux derivative and its characterization can be easily obtained by taking the limit $\lambda\to 0$ in the expression $y^\lambda:=\frac{y(v+\lambda v^\prime,\psi+\lambda\psi^\prime)-y(v,\psi)}{\lambda}$.
\end{proof}

With this lemma, we can verify conditions 1-4 of Proposition \ref{prop_saddle} for the cost functional \eqref{func_rob}.
\begin{proposition}\label{verif_cond}
Let $h\in L^2(\omega\times(0,T))$ and $y_0\in H^{-1}(\Omega)$ be given. There exists $\gamma$ large enough such that we have 
\begin{enumerate}
\item $\forall \psi\in L^2(Q)$, $v\mapsto J_r(v,\psi)$ is strictly convex and lower semicontinuous,
\item $\forall v\in L^2(\Sigma)$, $\psi\mapsto J_r(v,\psi)$ is strictly concave and upper semicontinuous,
\item $\lim_{\|v\|_{L^2(\Sigma)}\to+\infty}J_r(v,0)=+\infty$, 
\item $\lim_{\|\psi\|_{L^2(Q)}\to+\infty}J_r(0,\psi)=-\infty$.
\end{enumerate}
\end{proposition}

\begin{proof}
\textit{Condition 1.} By Lemma \ref{lemma_prop_sol1}, the map $v\mapsto J(v,\psi)$ is lower semicontinuous. Since $v\mapsto y(v,\psi)$ is linear and affine a straightforward computation yields the strict convexity of $J_r(v,\psi)$. 

\smallskip\noindent
\textit{Condition 2.} Again, by Lemma \ref{lemma_prop_sol1}, we deduce that the map $\psi\mapsto J(v,\psi)$ is upper semicontinuous. To prove strict concavity, we will argue as in \cite{temam_nonlinear}. Consider
\begin{equation*}
\mathcal{G}(\tau)=J_r(v,\psi+\tau \psi^\prime).
\end{equation*}
We will show that $\mathcal G(\tau)$ is sctrictly concave with respect to $\tau$, i.e. $\mathcal G^{\prime\prime}(\tau)<0$. We compute
\begin{equation*}
\mathcal G^\prime(\tau)=\iint_{\mathcal O_d\times(0,T)}(y+\tau y^\prime-y_d)y^\prime\dx\dt-\gamma^2\iint_{Q}(\psi+\tau\psi^\prime)\psi^\prime\dx\dt
\end{equation*}
where $y^\prime$ is solution to \eqref{deriv_sys1} with $v^\prime=0$. Since $y^\prime$ is clearly independent of $\tau$, a further computation yields
\begin{equation*}
\mathcal G^{\prime\prime}(\tau)=\iint_{\mathcal O_d\times(0,T)}|y^\prime|^2\dx\dt-\gamma^2\iint_{Q}|\psi^\prime|^2\dx\dt.
\end{equation*}
Using standard energy estimates for the heat equation, it is not difficult to see that 
\begin{equation*}
G^{\prime\prime}(\tau)\leq -(\gamma^2-C)\|\psi^\prime\|^2_{L^2(Q)}, \quad \forall \psi^\prime\in L^2(Q),
\end{equation*}
where $C=C(\Omega,\mathcal O_d,T)$ is a positive constant. We deduce that for a sufficiently large value of $\gamma$, we have $\mathcal G^{\prime\prime}(\tau)<0$, $\forall \tau\in \mathbb{R}$. Thus, the function $\mathcal G$ is striclty concave, and the strict concavity of $\psi$ follows immediately. 

\smallskip\noindent
\textit{Condition 3.} It is straightforward, taking $\psi=0$ in \eqref{func_rob} the result follows. 

\smallskip\noindent
\textit{Condition 4.} Since \eqref{heat_lin} is linear and using the estimate \eqref{ener_trans}, we have
\begin{equation*}
J(\psi,0)\leq \tilde C+C\|\psi\|_{L^2(Q)}^2-\frac{\gamma^2}{2}\|\psi\|_{L^2(Q)}^2
\end{equation*}
where $\tilde C$ is a positive constant only depending on $y_0$, $h$ and $y_d$. Hence, for $\gamma$ large enough, condition 4 holds. This concludes the proof.
\end{proof}

Combining the statements of Propositions \ref{prop_saddle} and \ref{verif_cond}, we are able to deduce the existence of at most one saddle point $(\bar v,\bar \psi)\in L^2(\Sigma)\times L^2(Q)$ for the cost functional \eqref{func_rob}. In particular, the existence of this saddle point implies that 
\begin{equation*}
\frac{\partial J_r}{\partial v}(\bar v,\bar \psi)=0 \quad \text{and}\quad \frac{\partial J_r}{\partial \psi}(\bar v,\bar \psi)=0.
\end{equation*}
By differentiating \eqref{func_rob}, we obtain the expressions 
\begin{align}\label{deriv_J1}
&\left(\frac{\partial J_r}{\partial v}(v,\psi),(v^\prime,0)\right)=\iint_{\mathcal O_d\times(0,T)}(y-y_d)y_v\,dxdt+\ell^2\iint_{\Sigma}vv^\prime\,d\sigma dt, \\ \label{deriv_J2}
&\left(\frac{\partial J_r}{\partial v}(v,\psi),(0,\psi^\prime)\right)=\iint_{\mathcal O_d\times(0,T)}(y-y_d)y_\psi\,dxdt-\gamma^2\iint_{Q}\psi\psi^\prime\,dxdt,
\end{align}
where we have denoted $y_v$ and $y_\psi$ the directional derivatives of $y$ solution to \eqref{heat_lin} (see Lemma \ref{lemma_prop_sol1}) in the directions $(v^\prime,0)$ and $(0,\psi^\prime)$, respectively. 

To characterize the solution to the robust control problem, we introduce the adjoint state to system \eqref{deriv_sys1}
\begin{equation}\label{adj_frontera}
\begin{cases}
-q_t-\Delta q=(y-y_d)\chi_{\mathcal O_d} \quad& \text{in }Q, \\
q=0 \quad &\textnormal{on $\Sigma$,} \\
q(x,T)=0 \quad &\textnormal{in $\Omega$}.
\end{cases}
\end{equation}

Multiplying \eqref{adj_frontera} by $y_v$ in $L^2(Q)$ and integrating by parts, we obtain  
\begin{equation*}
\iint_{\mathcal O_d\times(0,T)}(y-y_d)y_v\,dxdt+\iint_{\Sigma}\csbd v^\prime\frac{\partial q}{\partial n}\,d\sigma dt=0,
\end{equation*}
and upon substituting in \eqref{deriv_J1}, we get 
\begin{equation*}
\frac{\partial J_r}{\partial v}(v,\psi)=\frac{\partial q}{\partial n}\csbd-\ell^2v  
\end{equation*}
In a similar fashion, we multiply \eqref{adj_frontera} by $y_\psi$ in $L^2(Q)$ and integrate by parts. Then, from \eqref{deriv_J2} we deduce that 
\begin{equation*}
\frac{\partial J_r}{\partial \psi}(v,\psi)=q-\gamma^2\psi. 
\end{equation*}

Thus, so far, we have proved the following Proposition.
\begin{proposition}
Let $y_0\in H^{-1}(\Omega)$ and $h\in L^2(\omega\times(0,T))$ be given. Then
\begin{equation*}
\bar v=\frac{1}{\ell^2}\frac{\partial q}{\partial n}\csbd \quad\text{and}\quad \bar \psi=\frac{1}{\gamma^2}q
\end{equation*}
are the solution to the robust control problem stated in Definition \ref{defi_rob}, where $q$ is found from the solution $(y,q)$ to the optimality system 
\begin{equation}\label{sys_foll}
\begin{cases}
y_t-\Delta y= \chi_\omega h+\frac{1}{\gamma^2}q \quad&\textnormal{in }Q, \\
-q_t-\Delta q=(y-y_d)\chi_{\mathcal O_d} \quad& \textnormal{in }Q, \\
y=\frac{1}{\ell^2}\frac{\partial q}{\partial n}\csbd, \quad q=0 \quad &\textnormal{on $\Sigma$,} \\
y(x,0)=y_0(x), \quad q(x,T)=0 \quad &\textnormal{in $\Omega$}.
\end{cases}
\end{equation}
which admits a unique solution for $\gamma>0$ large enough. 
\end{proposition}

\begin{remark}
As in other robust control problems (cf. \cite{aziz,temam_nonlinear}), the characterization of the saddle point leads to a system of coupled equations. This characterization needs to be considered in the following step of the hierarchic control methodology, where a null control $h$ of minimal norm must be designed.
\end{remark}

\subsection{Null controllability}\label{sec_null_1}

In this section, we will proof an observability inequality that allows to establish the null controllability of system \eqref{sys_foll}. It is classical by now that null controllability is related to the observability of a proper adjoint system (see, for instance, \cite{cara_guerrero,zab}). 

For our particular case, let us consider the adjoint of system \eqref{sys_foll}:
\begin{equation}\label{adj_sys_foll}
\begin{cases}
-\varphi_t-\Delta \varphi= \theta\chi_{\mathcal O_d} \quad&\textnormal{in }Q, \\
\theta_t-\Delta \theta=\frac{1}{\gamma^2}\varphi \quad& \textnormal{in }Q, \\
\varphi=0, \quad \theta=\frac{1}{\ell^2}\frac{\partial \varphi}{\partial n}\csbd \quad &\textnormal{on $\Sigma$,} \\
\varphi(x,T)=\varphi^T(x), \quad \theta(x,0)=0 \quad &\textnormal{in $\Omega$}.
\end{cases}
\end{equation}
where $\varphi^T\in H^{1}_0(\Omega)$. Then, the controllability of \eqref{sys_foll} can be characterized in terms of appropriate properties of the solutions to \eqref{adj_sys_foll}. More precisely, we have the following proposition.

\begin{proposition}\label{prop_control}
Assume that $\omega\cap\mathcal O_d\neq \emptyset$ and $\ell,\gamma$ are large enough. The following statements are equivalent.
\begin{enumerate}
\item There exists a positive constant $C$, such that for any $y_0\in H^{-1}(\Omega)$ and any $y_d\in L^2(Q)$ such that \eqref{integ_yd} holds, there exists a control $h\in L^2(\omega\times(0,T))$ of minimal norm such that 
\begin{equation*}
\|h\|_{L^2(\omega\times(0,T))}\leq \sqrt C\left(\|y_0\|_{H^{-1}(\Omega)}+\|\rho y_d\|_{L^2(Q)}\right)
\end{equation*}
and the associated state satisfies $y(\cdot,T)=0$, where $y$ is the first component of \eqref{sys_foll}.  
\item There exist a positive constant $C$ and a weight function $\rho$ blowing up at $t=T$ such that the observability inequality 
\begin{equation}\label{obs_ineq_1}
\|\varphi(0)\|_{H^1_{0}(\Omega)}^2+\iint_Q \rho^{-2}|\theta|^2dxdt\leq C\iint_{\omega\times(0,T)}|\varphi|^2dxdt
\end{equation}
holds for every $\varphi^T\in H^1_{0}(\Omega)$, where $(\varphi,\theta)$ is the solution to \eqref{adj_sys_foll} associated to the initial datum $\varphi^T$. 
\end{enumerate}
\end{proposition}

Observe that the first statement implies Theorem \ref{teo_main1}. The equivalence between the statements of Proposition \ref{prop_control} is by now standard and relies on several classical arguments, so we omit the proof. The main idea follows from the well-known Hilbert uniqueness method (HUM for short), see, e.g., \cite{glo_lions_he,boyer_HUM}. For each $\varepsilon>0$, we consider the functional
\begin{equation}\label{func_h_leader}
F_\varepsilon(\varphi^T)=\frac{1}{2}\iint_{\omega\times(0,T)}|\varphi|^2\dx\dt+\frac{\varepsilon}{2}\|\varphi^T\|_{H_0^1(\Omega)}+\langle y_0,\varphi(0)\rangle_{-1,1}-\iint_{Q}\theta y_d\dx\dt,
\end{equation}
where $\langle\cdot,\cdot\rangle_{-1,1}$ denotes the duality product between $H^{-1}(\Omega)$ and $H_0^1(\Omega)$. One can prove that \eqref{func_h_leader} is continuous and strictly convex. Moreover, thanks to \eqref{obs_ineq_1} we can prove that $F_\varepsilon$ is coercive and therefore it has a unique minimizer that we denote as $\varphi_\varepsilon^T$. In this way, for each $\varepsilon>0$, we can construct controls $h_\varepsilon=\varphi_\varepsilon|_{\omega}$, where $\varphi_\varepsilon$ is the associated solution to \eqref{adj_sys_foll} with initial datum $\varphi_\varepsilon^T$, such that the controlled system \eqref{sys_foll} verifies $\|y(T)\|_{H^{-1}(\Omega)}\leq \varepsilon$. This fact, together with some uniform estimates on $h_\varepsilon$, enable us to obtain a null control as a limit process as $\varepsilon\to 0$. For the interested reader, we refer to \cite{araruna,vhs_deT_rob} where a detailed exposition is addressed for similar hierarchic control problems. 

For this reason, in what follows, we only focus on proving the estimate \eqref{obs_ineq_1}. To do so, let us introduce the Hilbert space
\begin{equation*}
W(Q):=\left\{u\in L^2(0,T;H^1(\Omega)), \ u_t\in L^2(0,T;H^{-1}(\Omega)) \right\},
\end{equation*}
equipped with the norm
\begin{equation*}
\|u\|_{W(Q)}=\left(\|u\|_{L^2(0,T;H^1(\Omega))}^2+\|u_t\|^2_{L^2(0,T;H^{-1}(\Omega))}\right)^{1/2},
\end{equation*}
and, for $r,s\geq 0$, we denote the Hilbert spaces $H^{r,s}(Q)=L^2(0,T;H^r(\Omega))\cap H^s(0,T;L^2(\Omega))$ endowed with  norms
\begin{equation}\label{norm_hrs}
\|u\|_{H^{r,s}(Q)}=\left(\|u\|^2_{L^2(0,T;H^r(\Omega))}+\|u\|_{H^s(0,T;L^2(\Omega))}^2\right)^{1/2}.
\end{equation}
\begin{remark}
We shall replace $Q$ and $\Omega$ by $\Sigma$ and $\partial\Omega$ in \eqref{norm_hrs} to refer to the analogous space for functions defined on the boundary.  
\end{remark}

The observability inequality \eqref{obs_ineq_1} will be consequence of global Carleman inequalities for parabolic equations with non-homogenous boundary conditions obtained in \cite{ima_yama_boundary}. To apply them, we need first to prove the following regularity result for the solutions to \eqref{adj_sys_foll}.

\begin{proposition}\label{prof_regularity}
Assume that $\varphi^T\in H_0^1(\Omega)$ and that $\ell,\gamma$ are large enough. Then, \eqref{adj_sys_foll} admits a unique solution $(\varphi,\theta)\in H^{2,1}(Q)\times W(Q)$. Moreover, $\frac{\partial \varphi}{\partial n}\in H^{\frac12,\frac14}(\Sigma)$. 
\end{proposition}
\begin{proof}
For any arbitrary $\bar \theta\in L^2(Q)$, let us consider the system
\begin{equation}\label{heat_fixed}
\begin{cases}
-\varphi_t-\Delta \varphi=\bar\theta\chi_{\mathcal O_d}, & \text{in Q}, \\
\varphi=0 &\text{on } \Sigma, \\
\varphi(x,T)=\varphi^T(x), & \text{in } \Omega.
\end{cases}
\end{equation}
Then, from classical regularity results (see, for instance, \cite{evans}), we have that 
\begin{equation*}
\varphi\in L^2(0,T;H^2(\Omega))\cap L^\infty(0,T;H_0^1(\Omega)), \quad \varphi_t\in L^2(Q).
\end{equation*}
This implies that $\varphi\in H^{2,1}(Q)$. Moreover, there exists a positive constant $C$ such that
\begin{equation}\label{norm_varphi_h21}
\|\varphi\|_{H^{2,1}(Q)}\leq C\left(\|\bar \theta\|_{L^2(Q)}+\|\varphi^T\|_{H_{0}^{1}(\Omega)} \right).
\end{equation}
Thanks to the regularity of $\varphi$, we have from \cite[Th. 2.1, p.9]{lions_magenes} that 
\begin{equation*}
\frac{\partial\varphi}{\partial n}\in H^{\frac12,\frac14}(\Sigma),
\end{equation*}
and, moreover,  $\varphi \mapsto \frac{\partial \varphi }{\partial n}$ is a continuous and linear map of $H^{\frac12,\frac14}(\Sigma)\to H^{2,1}(Q)$. Thus, we get
\begin{equation}\label{est_phi_h1214}
\|\tfrac{\partial \varphi}{\partial n}\csbd\|_{H^{\frac12,\frac14}(\Sigma)}\leq C\left(\|\bar \theta\|_{L^2(Q)}+\|\varphi^T\|_{H_{0}^{1}(\Omega)} \right)
\end{equation}
where we have used the fact that $\Gamma\subset \partial \Omega$ and $\csbd$ is a smooth function. 

On the other hand, for $\varphi$ solution to \eqref{heat_fixed}, let us consider the system
\begin{equation}\label{theta_fixed}
\begin{cases}
\theta_t-\Delta \theta=\frac{1}{\gamma^2}\varphi, & \text{in Q}, \\
\theta=\frac{1}{\ell^2}\frac{\partial \varphi}{\partial n}\csbd &\text{on } \Sigma, \\
\theta(x,0)=0, & \text{in } \Omega.
\end{cases}
\end{equation}
From the results in \cite[Ch. 4, \S15.5]{lions_magenes} and estimates \eqref{norm_varphi_h21}--\eqref{est_phi_h1214}, we deduce the existence of $C>0$ such that
\begin{align}\notag
\|\theta\|_{W(Q)}&\leq C\left(\frac{1}{\gamma^2}\|\varphi\|_{L^2(Q)}+\frac{1}{\ell^2}\|\tfrac{\partial \varphi}{\partial n}\csbd \|_{H^{\frac12,\frac14}(\Sigma)}\right) \\ \label{est_theta_fixed}
&\leq \frac{C}{\mu}\left(\|\bar\theta\|_{L^2(Q)}+\|\varphi^T\|_{H^{1}_0(\Omega)}\right),
\end{align}
where we have denoted $\mu=\min\{\ell^2,\gamma^2\}$.

Now, consider the mapping defined by $\Lambda\bar\theta=\theta$ where $(\varphi,\theta)$ is the solution to the cascade system composed by \eqref{heat_fixed} and \eqref{theta_fixed}. From the linearity of the systems, estimate \eqref{est_theta_fixed} and the embedding $W(Q)\hookrightarrow L^2(Q)$ (which follows from classical compactness results, see, e.g., \cite{simon}) we can check that $\Lambda$ is a well-defined, continuous and linear map of $L^2(Q)\to L^2(Q)$.

For any $\bar\theta_1,\bar\theta_2\in L^2(Q)$, we consider the difference $\Lambda\bar\theta_1-\Lambda\bar\theta_2$. Then, arguing as above, we can readily see that
\begin{equation*}
\|\Lambda \bar\theta_1-\Lambda \bar\theta_2\|_{L^2(Q)}\leq \frac{C}{\mu}\left(\|\bar\theta_1-\bar\theta_2\|_{L^2(Q)}\right).
\end{equation*}
For $\gamma,\ell$ large enough, the map $\Lambda$ is a contraction and therefore possesess a unique fixed point. The proof is complete. 
\end{proof}

\subsubsection*{The Carleman estimate}

Now, we are in position to prove one of the main results of this section. We begin by introducing several weight functions that will be useful in the remainder of this section. By hypothesis $\omega\cap\mathcal O_d\neq \emptyset$, then there exists a nonempty open set $\omega_0\subset\subset\omega\cap\mathcal O_d$. Let $\eta^0\in C^2(\overline\Omega)$ be a function verifying
\begin{equation}\label{constr_1}
\eta^0(x)>0 \ \text{in} \ \Omega, \quad \eta^0=0 \ \text{on} \ \Gamma, \quad |\nabla\eta^0|>0 \ \text{in} \ \overline{\Omega}\backslash\omega_0.
\end{equation}

The existence of such function is given in \cite{fursi}. Let $l\in C^\infty([0,T])$ be a positive function in $(0,T)$ satisfying 
\begin{equation*}
\begin{split}
&l(t)=t \quad \text{if } t\in [0,T/4], \quad l(t)=T-t \quad \text{if } t\in [3T/4,T], \\
&l(t)\leq l(T/2), \quad \text{for all } t\in[0,T]. 
\end{split}
\end{equation*}
Then, for all $\lambda\geq 1$ and $m\geq 2$, we consider the following weight functions:
\begin{equation}\label{weights_l} 
\begin{split}
&\alpha(x,t)= \frac{e^{2\lambda\|\eta^0\|_\infty}-e^{\lambda\eta^0(x)}}{l^m(t)}, \quad \xi(x,t)=\frac{e^{\lambda \eta^0(x)}}{l^m(t)} \\
&\alpha^*(t)=\max_{x\in\overline{\Omega}}\alpha(x,t), \quad \xi^*(t)=\min_{x\in\overline\Omega} \xi(x,t).
\end{split}
\end{equation}
The following notation will be used to abridge the estimates
\begin{equation*} 
\begin{gathered}
I(s;u):=s^{-1}\iint_{Q}e^{-2s\alpha}\xi^{-1}(|u_t|^2+|\Delta u|^2)\dx\dt+s\iint_Qe^{-2s\alpha}\xi|\nabla u|^2\dx\dt+s^3\iint_Qe^{-2s\alpha}\xi^3|u|^2\dx\dt \\
\widetilde I(s;u):=s^{-1}\iint_{Q}e^{-2s\alpha}\xi^{-1}|\nabla u|^2\dx\dt+s\iint_{Q}e^{-2s\alpha}\xi |u|^2\dx\dt,
\end{gathered}
\end{equation*}
for some parameter $s>0$. 

We state a Carleman estimate, due to \cite{ima_yama_boundary}, which holds for the solutions of heat equations with non-homogeneous Dirichlet boundary conditions.
\begin{lemma}
Let us assume $u_0\in L^2(\Omega)$, $f\in L^2(Q)$ and $g\in H^{\frac{1}{2},\frac{1}{4}}(\Sigma)$. Then, there exists a constant $\lambda_0$, such that for any $\lambda\geq \lambda_0$ there exist constants $C>0$ independent of $s$ and $s_0(\lambda)>0$, such that the solution $y\in W(Q)$ of 
\begin{equation*}
\begin{cases}
u_t-\Delta u=f \quad& \textnormal{in } Q, \\
u=g \quad& \textnormal{on } \Sigma, \\
u(0)=u_0 \quad& \textnormal{in } \Omega, 
\end{cases}
\end{equation*}
satisfies
\begin{equation}\label{car_boundary}
\begin{split}
\widetilde I(s;u)\leq C&\left(s^{-\frac{1}{2}}\|e^{-s\alpha}\xi^{-\frac{1}{4}}g \|_{H^{\frac{1}{2},\frac{1}{4}}(\Sigma)}^2+s^{-\frac{1}{2}}\|e^{-s\alpha}\xi^{-\frac{1}{4}+\frac{1}{m}}g\|^2_{L^2(\Sigma)}\phantom{\iint_{\omega\times(0,T)}}\right. \\
&\left.\quad+s^{-2}\iint_Qe^{-2s\alpha}\xi^{-2}|f|^2\dx\dt+s\iint_{\omega_0\times(0,T)}e^{-2s\alpha}\xi|u|^2\dx\dt\right)
\end{split}
\end{equation}
for every $s\geq s_0(\lambda)$. 
\end{lemma}

The second result we need is the classical Carleman estimate for the linear heat equation (see, for instance, \cite{fursi,cara_guerrero}).
\begin{lemma}\label{2once}
Let us assume $u^T\in H_0^1(\Omega)$ and $f\in L^2(Q)$. Then, there exists a constant $\lambda_1$, such that for any $\lambda\geq \lambda_1$ there exist constants $C>0$ independent of $s$ and $s_1(\lambda)>0$, such that the solution of 
\begin{equation}\label{sys_u}
\begin{cases}
u_t+\Delta u=f \quad& \textnormal{in } Q, \\
u=0 \quad& \textnormal{on } \Sigma, \\
u(T)=u^T(x) \quad& \textnormal{in } \Omega,
\end{cases}
\end{equation}
satisfies
\begin{equation}\label{car_clasica}
\begin{split}
I(s;u)\leq C&\left(\iint_Qe^{-2s\alpha}|f|^2\dx\dt+s^3\iint_{\omega_0\times(0,T)}e^{-2s\alpha}\xi^3|u|^2\dx\dt\right)
\end{split}
\end{equation}
for every $s\geq C$. 
\end{lemma}
\begin{remark}
In \cite{fursi,cara_guerrero}, the authors use the function $l(t)=t(T-t)$ to prove the above lemma. Since we propose $l(t)$ with the important property of going to $0$ as $t$ tends to $0$ and $T$, the proof of lemma \ref{2once} does not change if we use the weights in \eqref{weights_l}.
\end{remark}

Thanks to Lemma \ref{prof_regularity} and using the above estimates, we can prove a Carleman inequality for the solutions of the adjoint system \eqref{adj_sys_foll}. This will be the main ingredient to prove the observability inequality \eqref{obs_ineq_1}. The result is the following.
\begin{proposition}\label{prop_carleman}
Assume that $\omega\cap\mathcal O_d\neq \emptyset$ and $\gamma,\ell$ are large enough. Then, there exists a constant $C$ such that for any $\varphi^T\in H^1_0(\Omega)$, the solution $(\varphi,\theta)$ to \eqref{adj_sys_foll} satisfies
\begin{equation}\label{car_final}
\begin{split}
I(s;\varphi)+\widetilde I(s;\theta) \leq C\left(s^5\iint_{\omega\times(0,T)}e^{-2s\alpha}\xi^5|\varphi|^2\dx\dt\right),
\end{split}
\end{equation}
for every $s>0$ large enough.
\end{proposition}

\begin{proof}
We start by applying inequality \eqref{car_clasica} to the first equation in \eqref{adj_sys_foll} and inequality \eqref{car_boundary} to the second one. We take $\hat \lambda=\max\{\lambda_0,\lambda_1\}$ and fix $\lambda\geq \hat \lambda$. Adding them up, we obtain
\begin{equation*}
\begin{split}
I(s;\varphi)+\widetilde{I}(s;\theta)\leq C&\left(\iint_{\omega_0\times(0,T)}e^{-2s\alpha}\left(s^3\xi^3|\varphi|^2+s\xi|\theta|^2\right)\dx\dt +\iint_Q e^{-2s\alpha}\left(|\theta\chi_{\mathcal O_d}|^2+s^{-2}\xi^{-2}|\tfrac{1}{\gamma^2}\varphi|^2\right)\dx\dt\right. \\
&\quad\left. +s^{-\frac{1}{2}}\left\|e^{-s\alpha}\xi^{-\frac{1}{4}}\frac{1}{\ell^2}\frac{\partial \varphi}{\partial n}\csbd \right\|^2_{H^{\frac{1}{4},\frac{1}{2}}(\Sigma)}+s^{-\frac{1}{2}}\left\| e^{-2s\alpha}\xi^{-\frac{1}{4}+\frac{1}{m}} \frac{1}{\ell^2}\frac{\partial \varphi}{\partial n} \csbd \right\|^{2}_{L^2(\Sigma)}  \right),
\end{split}
\end{equation*}
for all $s$ large enough and some $m\geq 2$ to be fixed later. We can use the parameter $s$ to absorb the lower order terms into the left-hand side in the previous inequality. More precisely, we have 
\begin{equation}\label{car_con_bound}
\begin{split}
I(s;\varphi)+\widetilde{I}(s;\theta)\leq C&\left(\iint_{\omega_0\times(0,T)}e^{-2s\alpha}\left(s^3\xi^3|\varphi|^2+s\xi|\theta|^2\right)\dx\dt + s^{-\frac{1}{2}}\left\| e^{-2s\alpha}\xi^{-\frac{1}{4}+\frac{1}{m}} \frac{1}{\ell^2}\frac{\partial \varphi}{\partial n} \csbd \right\|^{2}_{L^2(\Sigma)} \right. \\
&\quad\left. +s^{-\frac{1}{2}}\left\|e^{-s\alpha}\xi^{-\frac{1}{4}}\frac{1}{\ell^2}\frac{\partial \varphi}{\partial n}\csbd \right\|^2_{H^{\frac{1}{4},\frac{1}{2}}(\Sigma)}  \right),
\end{split}
\end{equation}
for all $s\geq C$, with $C>0$ only depending on $\Omega$, $\omega$, $\mathcal O_d$ and $\lambda$. 

Then, we proceed to estimate the first boundary term in \eqref{car_con_bound}. Arguing as in \cite{duprez_lissy}, we take a function $\kappa\in C^2(\overline \Omega)$ such that
\begin{equation*}
\frac{\partial \kappa}{\partial \nu}=1 \quad\text{and}\quad \kappa=1 \quad \text{on } \partial \Omega.
\end{equation*}
Note that from the definition of the weight functions \eqref{weights_l}, we have $\alpha$ and $\alpha^*$ are equal in $\partial \Omega$, thus we may write
\begin{equation*}
\begin{split}
\iint_{\Sigma}e^{-2s\alpha}\left|\frac{\partial \varphi}{\partial n}\right|^2\d\sigma\dt&=\int_{0}^{T}e^{-2s\alpha^*}\int_{\partial \Omega}\nabla\varphi\cdot \nabla \kappa\frac{\partial \varphi}{\partial n} \d\sigma \dt\\
&=\int_{0}^{T}e^{-2s\alpha^*}\int_{\Omega}\Delta \varphi\,(\nabla\varphi\cdot \nabla \kappa)\dx\dt+\int_{0}^Te^{-2s\alpha^*}\int_\Omega\nabla (\nabla\varphi\cdot\nabla\kappa)\cdot\nabla \varphi \dx\dt,
\end{split}
\end{equation*}
where we have integrated by parts in the right-hand side. Using Cauchy-Schwarz and Young inequalities, we get
\begin{equation}\label{estimate_1}
\begin{split}
\iint_{\Sigma}&e^{-2s\alpha}\left|\frac{\partial \varphi}{\partial n}\right|^2\d\sigma \dt\\
&\leq C\left(\iint_Q e^{-2s\alpha^*}\left((s\xi)^{-1}|\Delta \varphi|^2 + s\xi|\nabla \varphi|^2\right)\dx\dt+\int_{0}^{T}e^{-2s\alpha^*}(s\xi^*)^{-1}\|\varphi\|^2_{H^2(\Omega)}\dt\right).
\end{split}
\end{equation}

We proceed to estimate the last term in the above inequality. Let us set $\widehat \varphi=\sigma\varphi$ with $\sigma\in C^\infty([0,T])$ defined as
\begin{equation}\label{rho_def}
\sigma:=e^{-s\alpha^*}(s\xi^*)^a
\end{equation}
for some $a\in \mathbb R$ to be chosen later. Observe that $\sigma(T)=0$. Then, $\widehat{\varphi}$ is solution to the system
\begin{equation}\label{phi_gorro}
\begin{cases}
-\widehat\varphi_t-\Delta \widehat\varphi=\sigma\,\theta\chi_{\mathcal O_d}-\sigma_t\varphi \quad& \text{in } Q, \\
\widehat\varphi=0 \quad& \text{on } \Sigma, \\
\widehat{\varphi}(\cdot,T)=0 \quad&\text{in } \Omega.
\end{cases}
\end{equation}

From standard regularity estimates for the heat equation, we have that $\widehat{\varphi}$ solution to \eqref{phi_gorro} satisfies
\begin{equation}\label{est_regul}
\|\widehat{\varphi}\|_{H^{2,1}(Q)}\leq C\left(\|\sigma\,\theta\|_{L^2(Q)}+\|\sigma_t\varphi\|_{L^2(Q)}\right).
\end{equation}
Using the definitions of $\alpha^*$ and $\xi^*$ given in \eqref{weights_l} together with \eqref{rho_def}, it is not difficult to see that 
\begin{equation*}
|\sigma_t|\leq Ce^{-s\alpha^*}(s\xi^*)^{a+{(m+1)}/{m}},
\end{equation*}
for all $s$ large enough. Employing  \eqref{est_regul} with $\sigma=e^{-s\alpha^*}(s\xi^*)^{-1/2}$, we obtain
\begin{equation}\label{estimate_2}
\begin{split}
\int_{0}^T e^{-2s\alpha^*}&(s\xi^{*})^{-1}\|\varphi\|_{H^2(\Omega)}^2\dt \\
&\leq C\left(\iint_Qe^{-2s\alpha^*}(s\xi^*)^{-1}|\theta|^2\dx\dt+\iint_Qe^{-2s\alpha^*}(s\xi^*)^{-1+2(m+1)/m}|\varphi|^2\dx\dt\right).
\end{split}
\end{equation}

Setting $m=4$ and combining estimates \eqref{estimate_1} and \eqref{estimate_2}, together with the definitions of the weights $\alpha^*$ and $\xi^*$ (see \eqref{weights_l}), yield
\begin{equation*}
\begin{split}
\iint_{\Sigma}e^{-2s\alpha}&\left|\frac{\partial \varphi}{\partial n}\right|^2\d\sigma\dt\\
&\leq C\left(\iint_Q e^{-2s\alpha}\left((s\xi)^{-1}|\Delta \varphi|^2 + s\xi|\nabla \varphi|^2\right)\dx\dt+\iint_Qe^{-2s\alpha}\left(s\xi|\theta|^2+(s\xi)^{3/2}|\varphi|^2\right)\dx\dt\right).
\end{split}
\end{equation*}
Thus, $\|e^{-s\alpha}\frac{\partial \varphi}{\partial \nu}\|_{L^2(\Sigma)}^2$ is bounded by the left-hand side of \eqref{car_con_bound} and, by taking $s$ large enough, we can now absorb the boundary term 
\begin{equation*}
s^{-\frac{1}{2}}\left\| e^{-2s\alpha} \frac{1}{\ell^2}\frac{\partial \varphi}{\partial n} \csbd \right\|^{2}_{L^2(\Sigma)}.
\end{equation*}

To deal with the second boundary term in \eqref{car_con_bound}, we can use the facts that $\varphi \mapsto \frac{\partial \varphi }{\partial n}$ is a continuous and linear map of $H^{\frac12,\frac14}(\Sigma)\to H^{2,1}(Q)$ and $\csbd$ is a smooth function. More precisely. 

\begin{equation}\label{est_trace}
\begin{split}
s^{-\frac{1}{2}}\left\|e^{-s\alpha}\xi^{-\frac{1}{4}} \frac{1}{\ell^2} \frac{\partial \varphi}{\partial n}\csbd \right\|^2_{H^{\frac{1}{2},\frac{1}{4}}(\Sigma)} & \leq s^{-\frac{1}{2}}\left\|e^{-s\alpha^*}(\xi^*)^{-\frac{1}{4}} \frac{\partial \varphi}{\partial n} \right\|^2_{H^{\frac{1}{2},\frac{1}{4}}(\Sigma)} \\
&\leq s^{-\frac{1}{2}}\left\| e^{-s\alpha^*}(\xi^*)^{-\frac{1}{4}} \varphi \right\|_{H^{2,1}(Q)},
\end{split}
\end{equation}
where we have used again that $\alpha^*,\xi^*$ coincide with $\alpha,\xi$ on the boundary. 

Now, set $\widehat \varphi=\sigma \varphi$ with $\sigma=e^{-s\alpha^*}(s\xi)^{-\frac{1}{4}}$. Arguing as before, we readily obtain from \eqref{est_regul} the following 
\begin{equation}\label{est_frontera}
s^{-\frac{1}{2}}\left\| e^{-s\alpha^*}(\xi^*)^{-\frac{1}{4}} \varphi \right\|_{H^{2,1}(Q)}\leq C\left(\iint_{Q}e^{-2s\alpha^*}(s\xi^*)^{-\frac{1}{2}}|\theta|^2\dx\dt+\iint_{Q}e^{-2s\alpha^*}(s\xi^*)^2|\varphi|^2\dx\dt\right).
\end{equation}
Putting together \eqref{est_trace} and \eqref{est_frontera} and substituting in \eqref{car_con_bound}, we can take $s$ sufficiently large to absorb the remaining terms. 

Up to now, we have 
\begin{equation}\label{est_locales}
\begin{split}
I(s;\varphi)+\widetilde{I}(s;\theta)\leq C&\left(\iint_{\omega_0\times(0,T)}e^{-2s\alpha}\left(s^3\xi^3|\varphi|^2+s\xi|\theta|^2\right)\dx\dt \right)
\end{split}
\end{equation}
for all $s$ large enough. The last step is to eliminate the local term corresponding to $\theta$. To this end, consider an open set $\omega_1$ such that $\omega_0\subset\subset\omega_1\subset \subset \omega\cap \mathcal O_d$ and a function $\zeta\in C_0^2(\omega_1)$ verifying 
\begin{equation}\label{cut_off}
\zeta\equiv 1\quad \text{in } \omega_0 \quad \text{and} \quad 0\leq \zeta\leq 1. 
\end{equation}
Thanks to the hypothesis $\omega\cap\mathcal O_d\neq \emptyset$, we can use the first equation of \eqref{adj_sys_foll} and \eqref{cut_off} to obtain
\begin{equation}\label{car_inter}
\iint_{\omega_0\times(0,T)}e^{-2s\alpha}s\xi|\theta|^2\dx\dt\leq \iint_{\omega_1\times(0,T)}e^{-2s\alpha}s\xi \theta(-\varphi_t-\Delta \varphi)\zeta\dx\dt.
\end{equation}
Following well-known arguments (see, e.g. \cite{luz_manuel,deteresa2000}), we integrate by parts in time and space several times in the right-hand side of \eqref{car_inter}. Then, using the expression of the second equation of \eqref{adj_sys_foll}, it is not difficult to see that 
\begin{equation}\label{inter_res}
\begin{split}
\iint_{\omega_0\times(0,T)}&e^{-2s\alpha}s\xi|\theta|^2\dx\dt\\
&\leq \frac{1}{\gamma^2}\iint_{Q}e^{-2s\alpha}s\xi |\varphi|^2\dx\dt+\iint_{\omega_1\times(0,T)}e^{-2s\alpha}\left(s^3\xi^3|\theta| |\varphi|+s^2\xi^2|\nabla \theta||\varphi|\right)\dx\dt,
\end{split}
\end{equation}
where we have used that for any $s$ large enough, the following inequalities hold
\begin{equation*}
|\Delta(\zeta e^{-2s\alpha}\xi)|\leq Cs^2e^{-2s\alpha}\xi^3 \quad \text{and}\quad |(e^{-2s\alpha}\xi)_t|\leq Cse^{-2s\alpha}(\xi)^{2+\frac{1}{m}}.
\end{equation*}

Using H\"older and Young inequalities in the last term of \eqref{inter_res}, it follows that
\begin{equation}\label{est_final_1}
\begin{split}
\iint_{\omega_0\times(0,T)}e^{-2s\alpha}s\xi|\theta|^2\dx\dt\leq& \frac{1}{\gamma^2}\iint_{Q}e^{-2s\alpha}s\xi |\varphi|^2\dx\dt+\delta \widetilde I(s;\theta)+C_{\delta}s^5\iint_{\omega_1\times(0,T)}e^{-2s\alpha}\xi^5|\varphi|^2\dx\dt.
\end{split}
\end{equation}
for any $\delta>0$. We replace \eqref{est_final_1} in \eqref{est_locales} with $\delta$ small enough and  since $\gamma$ is sufficiently large, we can absorb the first two terms of the above inequality. Finally, noting that $\omega_1\subset \omega$, we obtain the desired inequality. This concludes the proof of Proposition \ref{prop_carleman}
\end{proof}

\subsubsection*{The observability inequality}

The observability inequality \eqref{obs_ineq_1} follows from the Carleman estimate \eqref{car_final}, some weighted energy estimates and it is based on well-known arguments which can be adapted from \cite[Lemma 1]{cara_NS}. For the sake of completeness, we sketch it briefly. 

Let us introduce the weight functions:
\begin{equation}\label{weights_rec} 
\begin{split}
&\beta(x,t)= \frac{e^{2\lambda\|\eta^0\|_\infty}-e^{\lambda\eta^0(x)}}{\bar l^4(t)}, \quad \phi(x,t)=\frac{e^{\lambda \eta^0(x)}}{\bar l^4(t)} \\
&\beta^*(t)=\max_{x\in\overline{\Omega}}\beta(x,t), \quad \phi^*(t)=\min_{x\in\overline\Omega} \phi(x,t),
\end{split}
\end{equation}
where 
\begin{equation*}
\bar l(t)=
\begin{cases}
\|l\|_\infty &\quad\text{for}\quad 0\leq t\leq T/2, \\
l(t) &\quad\text{for}\quad T/2\leq t\leq T. 
\end{cases}
\end{equation*}

We set $s$ to a sufficiently large fixed value. By construction, $\alpha=\beta$ in $[0,T/2]$. Thus, 
\begin{align} \notag 
&\iint_{\Omega\times(T/2,T)}e^{-2s\beta}\phi^3|\varphi|^2\dx\dt+\iint_{\Omega\times(T/2,T)}e^{-2s\beta}\phi|\nabla \varphi|^2\dx\dt+\iint_{\Omega\times(T/2,T)}e^{-2s\beta^*}|\theta|^2\dx\dt  \\ \notag
&\quad = \iint_{\Omega\times(T/2,T)}e^{-2s\alpha}\xi^3|\varphi|^2\dx\dt + \iint_{\Omega\times(T/2,T)}e^{-2s\alpha}\xi|\nabla\varphi|^2\dx\dt + \iint_{\Omega\times(T/2,T)}e^{-2s\alpha^*}|\theta|^2\dx\dt \\ \label{car_rec}
&\quad \leq  C(s,T)\left(\iint_{\omega\times(0,T)}e^{-2s\beta}\phi^5|\varphi|^2\dx\dt\right)
\end{align}
where we have used the Carleman inequality \eqref{car_final}, the weight properties \eqref{weights_l}, \eqref{weights_rec} and the fact that $\beta\leq \alpha$ in Q. 

On the other hand, consider a function $\tilde \eta\in C^1([0,T])$, such that 
\begin{equation*}
\tilde\eta \equiv 1 \quad\text{in } [0,T/2], \quad \tilde\eta\equiv 0 \quad \text{in } [3T/4,T], \quad |\tilde\eta^\prime|\leq C/T.
\end{equation*}
Let us multiply by $\tilde\eta\Delta\varphi$ the equation verified by $\varphi$ (see \eqref{adj_sys_foll}) and integrate in $\Omega$, namely 
\begin{equation*}
-\int_{\Omega}\varphi_t\Delta\varphi\tilde\eta\dx-\int_{\Omega}|\Delta \varphi|^2\tilde\eta\dx=\int_{\mathcal O_d}\theta\Delta\varphi\tilde \eta\dx
\end{equation*}
Integrating by parts and using H\"older and Young inequalities yield
\begin{equation*}
-\frac{1}{2}\frac{\d}{\dt}\int_{\Omega}|\nabla\varphi|^2\tilde\eta\dx+\int_{\Omega}|\Delta\varphi|^2\tilde\eta\dx\leq \delta \int_{\Omega}|\Delta\varphi|^2\tilde\eta\dx+C_\delta\int_{\Omega}|\theta|^2\tilde\eta\dx+\frac{1}{2}\int_{\Omega}|\nabla\varphi|^2|\eta^\prime|\dx
\end{equation*}
for any $\delta>0$.  Choosing $\delta$ small enough and dropping the positive term involving $|\Delta\varphi|^2\tilde\eta$, we integrate in $[0,T]$, from which we obtain
\begin{equation}\label{nabla_theta_0}
\int_{\Omega}|\nabla\varphi(0)|^2\dx\leq C\left(\iint_{\Omega\times(0,3T/4)}|\theta|^2\dx\dt+\frac{1}{T}\iint_{\Omega\times(T/2,3T/4)}|\nabla\varphi|^2\dx\dt\right).
\end{equation}
Here we have used the properties of the function $\tilde\eta$. 

Using the definition of the weight $\beta$, it can be readily seen that
\begin{equation*}
e^{-2s\beta}\phi\geq C >0 \quad\text{and}\quad \quad  e^{-2s\beta^*}\geq C>0, \quad \forall t\in[T/2,3T/4],
\end{equation*}
and therefore we can bound appropriately in \eqref{nabla_theta_0} to obtain
\begin{equation*}
\int_{\Omega}|\nabla\varphi(0)|^2\dx\leq C\left(\iint_{\Omega\times(0,T/2)}|\theta|^2\dx\dt+\iint_{\Omega\times(T/2,3T/4)}\left(e^{-2s\beta}\phi|\nabla\varphi|^2+e^{-2s\beta^\star}|\theta|^2\right)\dx\dt\right).
\end{equation*}
Observe that the second term in the above expression can be estimated with inequality \eqref{car_rec}, this is
\begin{equation}\label{est_h01_theta}
\int_{\Omega}|\nabla\varphi(0)|^2\dx\leq C\left(\iint_{\Omega\times(0,T/2)}|\theta|^2\dx\dt+\iint_{\omega\times(0,T)}e^{-2s\beta}\phi^5|\varphi|^2\dx\dt\right).
\end{equation}

Using similar arguments, one may deduce an estimate of the form,
\begin{equation}\label{est_phi_0}
\iint_{\Omega\times(0,T/2)}e^{-2s\beta}\phi^3|\varphi|^2\dx\dt\leq C\left(\iint_{\Omega\times(0,T/2)}|\theta|^2\dx\dt+\iint_{\omega\times(0,T)}e^{-2s\beta}\phi^5|\varphi|^2\dx\dt\right),
\end{equation}
where we have used that $\beta$ and $\phi$ are bounded functions in $\Omega\times[0,T/2]$.

Arguing as in the proof of Lemma \ref{prof_regularity}, we can deduce the following estimate for the solutions to $\theta$ in the domain $\Omega\times [0,T/2]$
\begin{align}\notag 
\|\theta\|^2_{L^2(0,T/2;L^2(\Omega))}&\leq C\|\theta\|^2_{W(\Omega\times(0,T/2))}\\ \notag
&\leq C\left(\frac{1}{\gamma^4}\|\varphi\|^2_{L^2(\Omega\times(0,T/2))}+\frac{1}{\ell^4}\|\tfrac{\partial \varphi}{\partial n}\|^2_{H^{\frac12,\frac14}(\Omega\times(0,T/2))}\right) \\ \label{est_theta}
&\leq \frac{C}{\mu}\|\varphi\|_{H^{2,1}(\Omega\times(0,T/2))}^2,
\end{align}
where $\mu=\min\{\ell^4,\gamma^4\}$. Observe that in this domain, the function $\beta^\star$ is actually constant in both variables, therefore, from the above inequality we get
\begin{align}\notag
\iint_{\Omega\times(0,T/2)}e^{-2s\beta^*}|\theta|^2\dx\dt &\leq \frac{C}{\mu}\|e^{-s\beta^*}\varphi\|_{H^{2,1}(\Omega\times(0,T/2))}^2 \\ \label{est_theta_tme}
&\leq \frac{C}{\mu}\|e^{-s\beta^*}\varphi\|_{H^{2,1}(Q)}^2.
\end{align}

Combining \eqref{car_rec} with \eqref{est_h01_theta}--\eqref{est_theta_tme} we deduce
\begin{equation}\label{obs_mu}
\begin{split}
\|\varphi(0)\|_{H_0^1(\Omega)}^2+\iint_{Q}&e^{-2s\beta}\phi^3|\varphi|^2\dx\dt+\iint_{Q}e^{-2s\beta^*}|\theta|^2\dx\dt \\
&\leq \frac{C}{\mu}\|e^{-s\beta^*}\|_{H^{2,1}(Q)}^2+ C\iint_{\omega\times(0,T)}e^{-2s\beta}\phi^5|\varphi|^2\dx\dt.
\end{split}
\end{equation}

To conclude, it is enough to estimate the $H^{2,1}$-norm in the right-hand side of \eqref{obs_mu}. To do this, we use the system \eqref{phi_gorro} and estimate \eqref{est_regul} with $\sigma=e^{-s\beta^\star}$. More precisely, we have 
\begin{equation}\label{est_C_mu}
\|e^{-s\beta^\star}\|^2_{H^{2,1}(Q)}\leq C\left(\iint_{Q}e^{-2s\beta^\star}(\phi^\star)^{5/2}|\varphi|^2\dx\dt+\iint_{Q}e^{-2s\beta^\star}|\theta|^2\dx\dt\right).
\end{equation}
Finally, we obtain the observability inequality \eqref{obs_ineq_1} with $\varrho_1(t):=e^{s\beta^\star}$ by combining estimates \eqref{obs_mu}--\eqref{est_C_mu}, using the definition of $\beta^\star$ and $\phi^\star$ and taking $\gamma,\ell>>1$. This concludes the proof. 

\section{The case with distributed follower and boundary leader}\label{sec_bound_leader}

In this section, we address the robust hierarchic problem for the case when the leader is now placed on the boundary. More precisely, let us consider systems of the form
\begin{equation}\label{heat_dif1}
\begin{cases}
z_t-\Delta z=\csin{\mathcal B_1}v+\csin{\mathcal B_2}\psi, & \text{in Q}, \\
z=h\chi_{\cbd} &\text{on } \Sigma, \\
z(x,0)=z^0(x), & \text{in } \Omega.
\end{cases}
\end{equation}

It is well-known (see, e.g., \cite{ima_original,cara_guerrero}) that for $v\equiv\psi\equiv 0$, system \eqref{heat_dif1} is null controllable at time $T$ for any open set $\Gamma\subset\partial \Omega$. Here, we will see that using a hierarchic methodology for solving first the associated robust control problem introduces additional difficulties when dealing with the null controllability objective. 

Adapting the arguments of Section \ref{ex_uniq_saddle}, it is not difficult to prove the existence and uniqueness of the saddle point $(\bar v,\bar \psi)$ for the cost functional 
\begin{equation*}
K_r(v,\psi;h)=\frac{1}{2}\iint_{\mathcal O_d\times(0,T)}|y-y_d|^2\dx\dt+\frac{1}{2}\left[\ell^2\iint_{\mathcal B_1\times(0,T)}|v|^2\dx\dt-\gamma^2\iint_{\mathcal B_2\times(0,T)}|\psi|^2\dx\dt\right].
\end{equation*}
Indeed, for any given $h\in L^2(\cbd\times(0,T))$ and $y_0\in H^{-1}(\Omega)$, one can prove that
\begin{equation*}
\bar v=-\frac{1}{\ell^2}q|_{\mathcal B_1} \quad\text{and}\quad \bar \psi=\frac{1}{\gamma^2}q|_{\mathcal B_2}
\end{equation*}
are the solution to the robust control problem (see and adapt accordingly Definition \ref{defi_rob}), where $p$ is found from the solution $(z,p)$ to the optimality system 

\begin{equation}\label{couple1}
\begin{cases}
z_t-\Delta z=-\frac{1}{\ell^2}p\csin{\mathcal B_1}+\frac{1}{\gamma^2}p\csin{\mathcal B_2}, & \text{in Q}, \\
-p_t-\Delta p=(z-z_d)\chi_{\mathcal O_d}, & \text{in Q}, \\
z=h\csin{\Gamma}, \quad p=0 &\text{on } \Sigma, \\
z(x,0)=z^0(x), \quad p(x,T)=0& \text{in } \Omega.
\end{cases}
\end{equation}
Then, the problem amounts to study the null controllability of the coupled system \eqref{couple1}. 

As mentioned in the previous section, we can build the control $h$ of minimal norm by adapting the well-known HUM method. For this particular case, the main ingredient lies at finding, for each $\varepsilon>0$, the minimizers of the functional
\begin{equation}\label{func_FR}
F_{\epsilon}(\varphi^T)=\frac{1}{2}\iint_{\cbd\times(0,T)}\left|\frac{\partial \varphi}{\partial n}\right|^2\d\sigma\dt+\frac{\varepsilon}{2}\|\varphi^T\|_{H_0^1(\Omega)}+\langle z^0(x),\varphi(0)\rangle_{-1,1}-\iint_{Q}\theta z_d \dx\dt
\end{equation}
defined for the solutions to the adjoint system
\begin{equation}\label{adjunto1}
\begin{cases}
-\varphi_t-\Delta \varphi=\theta\csin{\mathcal \mathcal O_d}, & \text{in Q}, \\
\theta_t-\Delta \theta=-\frac{1}{\ell^2}\varphi\csin{\mathcal B_1}+\frac{1}{\gamma^2}\varphi \csin{\mathcal B_2}, & \text{in Q}, \\
\varphi=0, \quad \theta=0 &\text{on } \Sigma, \\
\varphi(x,T)=\varphi^T(x), \quad \theta(x,0)=0& \text{in } \Omega.
\end{cases}
\end{equation}

As before, we mainly focus proving an observability inequality to establish the coerciveness of the functional \eqref{func_FR} and thus guaranteeing the existence and uniqueness of the minimizers. We have the following result.
\begin{proposition}\label{prop_obs_ineq_dif}
Assume that \eqref{loc_teo2} holds and $\ell,\gamma$ are large enough.  Then, there exist a positive constant $C$ and a weight function $\vartheta$ blowing up at $t=T$ such that the observability inequality
\begin{equation}\label{obs_ineq_dif}
\left\|\varphi(x,0)\right\|^2_{H_0^1(\Omega)}+\iint_{Q}{\vartheta}_2^{-2}|\theta|^2dxdt\leq C\iint_{\cbd \times(0,T)}\left|\frac{\partial \varphi}{\partial n} \right|^2\d\sigma\dt
\end{equation}
holds for every $\varphi^T\in H_0^{1}(\Omega)$, where $(\varphi,\theta)$ is the solution to \eqref{adjunto1} associated to the inital datum $\varphi^T$. 
\end{proposition}

\begin{remark}
Some remarks are in order.
\begin{itemize}
\item To avoid unnecessary notation, we have decided to keep the letters $(\varphi,\theta)$ to denote the adjoint variables. Notice, however, that  systems \eqref{adj_sys_foll} and \eqref{adjunto1} are substantially different. 
\item The well-posedness and regularity of \eqref{adjunto1} follows from an argument similar to the one in Proposition \ref{prof_regularity}. Indeed, in this case, one may prove that for any $\varphi^T\in H_0^1(\Omega)$ and $\ell,\gamma$ large enough, system \eqref{adjunto1} admits a unique solution $(\varphi,\theta)\in H^{2,1}(Q)\times H^{2,1}(Q)$. 
\item From this regularity, we have that $\varphi\in L^2(0,T;H^2(\Omega))$ which, together with well-known trace estimates, justify that $\frac{\partial\varphi}{\partial n}\in L^2(\Sigma)$.
\end{itemize}
\end{remark}

As in the end of Section \ref{bound_follow}, the observability inequality \eqref{obs_ineq_dif} follows from a global Carleman estimate for the solutions to \eqref{adjunto1} and some weighted energy estimates. 

We begin by  introducing some useful weight functions. We set $\mathcal S^\prime\subset\subset \cbd$ and consider $\eta_{i}\in C^2(\overline \Omega)$, $i=1,2$, verifying
\begin{equation}\label{constr_bound}
\eta_{i}>0\quad \text{and}\quad |\nabla\eta_{i}|>0 \quad \text{in } \Omega, \quad \frac{\partial\eta_{i}}{\partial \eta}\leq 0 \quad \text{on } \partial \Omega\setminus S^\prime.
\end{equation}
This is the classical weight function introduced in \cite[Lemma 1.2]{ima_original} to prove a Carleman inequality with boundary observation for the heat equation and means that it should not have critical points inside the domain $\Omega$. In addition, we will choose the functions $\eta^1$ and $\eta^2$ in a way such that
\begin{gather}\label{props_brazil}
\eta_1\geq \eta_2 \quad\text{in }\Omega, \quad\eta_1=\eta_2\quad \text{in  }\mathcal O_d, \\ \label{prop_impor}
 \eta_1\geq \max_{x\in \bar{\mathcal B}_i} \eta_2 \quad \text{in } \bar{\mathcal B}_i, \quad i=1,2.
\end{gather}
These additional features where recently used in \cite{da_silva} for solving a similar hierarchic control problem. 
\begin{remark}
From property \eqref{props_brazil}, we can note that,
\begin{equation}\label{weights_12}
\tilde{\xi}_2\leq \tilde{\xi}_1,\quad \tilde\alpha_1\leq \tilde\alpha_2\quad\text{in } Q\quad\text{and}\quad\tilde{\xi}_2= \tilde{\xi}_1,\quad \tilde\alpha_1=\tilde\alpha_2\quad\text{in } \mathcal{O}_d\times(0,T), \quad i=1,2. 
\end{equation}
\end{remark} 

For all $\lambda\geq1$, let us consider the following weight functions
\begin{equation}\label{pesos_boundary}
\begin{split}
\tilde{\alpha}_i(x,t)=\frac{e^{\lambda(\|\eta_1\|+\|\eta_2\|)}-e^{\lambda \eta_i(x)}}{t^2(T-t)^2}, \qquad \tilde{\xi}_i(x,t)=&\frac{e^{\lambda\eta_i(x)}}{t^2(T-t)^2}, \quad i=1,2,
\end{split}\end{equation}
and for some parameters $s>0$ and $m\in\mathbb R$, we set the notation
\begin{equation*}
I(u;i,m):=\iint_{Q}e^{-2s\tilde{\alpha}_i}(s\tilde{\xi}_i)^{1+m}|\nabla u|^2\dx\dt+\iint_{Q}e^{-2s\tilde{\alpha}_i}(s\tilde{\xi}_i)^{3+m}|u|^2\dx\dt.
\end{equation*}

We have the following Lemma.
\begin{lemma}\label{lemma_car_boundary}
Assume that $u_0\in H_0^1(\Omega)$ and $f\in L^2(Q)$. For any $m\in \mathbb R$, there exists a constant $\lambda_m$, such that for any $\lambda\geq \lambda_m$ there exist constants $C>0$ independent of $s$ and $s_m(\lambda)>0$, such that the solution of \eqref{sys_u} satisfies
\begin{equation}\label{Est1}
I(u;i,m)\leq C\left(\iint_{Q}e^{-2s\tilde{\alpha}_i}(s\tilde{\xi}_i)^m|f|^2\dx\dt+\iint_{\mathcal S^\prime\times(0,T)}e^{-2s\tilde{\alpha}_i}(s\tilde{\xi}_i)^{1+m}\left|\frac{\partial u}{\partial n}\right|^2\d\sigma\dt\right)
\end{equation}
for every $s\geq s_m(\lambda)$. 
\end{lemma}
The proof of this result can be deduced from the Carleman inequality stated in \cite[Lemma 1.1]{ima_original} and adapting the arguments in \cite[Lemma 2.3]{ima_yama}. Now, we are in position to prove the observability inequality \eqref{obs_ineq_dif}.

\begin{proof}[Proof of Proposition \ref{prop_obs_ineq_dif}]
We apply estimate \eqref{Est1} to the each equation of the coupled system \eqref{adjunto1} with different values of $i$ and $m$ and add them up. In more detail, we have that
\begin{align}\notag
I(\varphi;1,0)&+I(\theta;2, -1)\\\notag
&\leq C\left(\iint_{\mathcal O_d\times(0,T)}e^{-2s\tilde{\alpha}_1}|\theta|^2\dx\dt+\iint_{\mathcal S^\prime\times(0,T)}e^{-2s\tilde{\alpha}_1}(s\tilde{\xi}_1)\left|\frac{\partial \varphi}{\partial n}\right|^2\d\sigma\dt\right.\\ \notag
&\quad+\left.\iint_{Q}e^{-2s\tilde{\alpha}_2}(s\tilde{\xi}_2)^{-1}|-\frac{1}{\ell^2}\varphi\csin{\mathcal B_1}+\frac{1}{\gamma^2}\varphi \csin{\mathcal B_2}|^2\dx\dt+\iint_{\mathcal S^\prime\times(0,T)}e^{-2s\tilde{\alpha}_2}\left|\frac{\partial \theta}{\partial n}\right|^2\d\sigma\dt\right)
\end{align}
holds for any $\lambda\geq\hat\lambda$, where $\hat{\lambda}=\max\{\lambda_0,\lambda_{-1}\}$ and any $s$ large enough.

From \eqref{weights_12}, we readily deduce that $e^{-2s\tilde\alpha_2}\leq e^{-2s\tilde\alpha_1}$ in $Q$ and that $\xi_2^{-1}\xi_1\leq C$, where $C>0$ only depends on $\Omega$ and $\mathcal S^\prime$. Moreover, $e^{-2s\tilde\alpha_1}=e^{-2s\tilde\alpha_2}$ in  $\mathcal O_d\times(0,T)$. In this way, we can absorb the lower order terms so 
\begin{align}\notag
I(\varphi;1,0)&+I(\theta;2, -1)\\ \label{igualda1}
&\leq C\left(\iint_{\mathcal S^{\prime}\times(0,T)}e^{-2s\tilde{\alpha}_1}(s\tilde{\xi}_1)\left|\frac{\partial \varphi}{\partial n}\right|^2\d\sigma\dt+\iint_{\mathcal S^\prime\times(0,T)}e^{-2s\tilde{\alpha}_2}\left|\frac{\partial \theta}{\partial n}\right|^2\d\sigma\dt\right),
\end{align}
holds for every $s$ sufficiently large.

 Notice that unlike the proof of Propositon \ref{prop_carleman}, we cannot obtain any local information from system \eqref{adjunto1} to estimate the last term in \eqref{igualda1}. Instead, we will use the particular selection of the weight functions $\eta_i$ to eliminate the observation of $\theta$ at the boundary.
 
 The first step is to improve \eqref{igualda1} in the sense that the weight functions do not vanish at $t=0$. Let us consider the function
 \begin{equation*}
 \widetilde{l}(t)=
 \begin{cases}
 T^4/16 \quad &\text{for}\quad 0\leq t\leq T/2, \\
 t^2(T-t)^2 \quad &\text{for}\quad T/2\leq t\leq T,
 \end{cases}
 \end{equation*}
 and the weights
 \begin{equation*}
\begin{split}
\tilde{\beta}_i(x,t)=\frac{e^{\lambda(\|\eta_1\|+\|\eta_2\|)}-e^{\lambda \eta_i(x)}}{\widetilde{l}(t)}, \qquad \tilde{\phi}_i(x,t)=&\frac{e^{\lambda\eta_i(x)}}{\widetilde{l}(t)}, \quad i=1,2.
\end{split}\end{equation*}
It can be proved that the following observability inequality holds:
\begin{equation}
\label{aux7}
\begin{split}
\|\varphi(x,0)\|_{H_0^1(\Omega)}^2+&\int_0^{T}\int_{\Omega}e^{-2s\tilde \beta_1}\tilde \phi_1^{3}|\varphi|^2\dx\dt+\int_0^{T}\int_{\Omega}e^{-2s\tilde \beta_2}\tilde \phi_2^{2}|\theta|^2\dx\dt\\
&\leq C\left(\iint_{\mathcal S^\prime\times(0,T)}e^{-2s\tilde \beta_1}\tilde \phi_1\left|\frac{\partial \varphi}{\partial n}\right|^2\d\sigma\dt+\iint_{\mathcal S^{\prime}\times(0,T)}e^{-2s\tilde \beta_2}\left|\frac{\partial \theta}{\partial n}\right|^2\d\sigma\dt\right),
\end{split}
\end{equation}
 for some universal constant $C>0$ depending on $s$ and $T$. Indeed, the proof follows the steps of Section \ref{sec_null_1}, but for the sake of brevity, we omit it. 
 
 Of course, this inequality still has the boundary observation of the variable $\theta$. Nevertheless, the definition of the weight $\tilde\beta_2$ and property \eqref{prop_impor} will help us to estimate it in terms of localized terms of $\varphi$ in $\mathcal B_1$ and $\mathcal B_2$. Hypotheses \eqref{loc_teo2} and property \eqref{prop_impor} imply that
 \begin{equation*}
 \eta_1(x) \geq \max_{x\in \mathcal S^\prime}\eta_2 \quad \forall  x\in\bar{\mathcal B}_i, \quad i=1,2.
 \end{equation*}
 Therefore,
 \begin{equation}\label{rel_eta1_eta2}
 e^{s\tilde\beta_1}\leq e^{s\widehat{\beta_2}} \quad\text{in } \bar{\mathcal B}_i\times(0,T), \quad i=1,2,
 \end{equation}
 where we have defined
 \begin{equation*}
 \widehat{\beta_2}(t):=\frac{e^{\lambda(\|\eta_1\|+\|\eta_2\|)}-e^{{\lambda}\max_{x\in\mathcal S^\prime}\eta_2}}{\widetilde{l}(t)}.
 \end{equation*}
Moreover, we can readily see that $e^{s\widehat{\beta_2}}\leq e^{s\tilde\beta_2}$ for all $(x,t)\in \mathcal S^\prime\times(0,T)$. This, together with the fact that $\mathcal S^\prime\subset\subset\partial \Omega$ and the well-known trace estimate, allow us to obtain 
\begin{align}\notag 
\iint_{S^{\prime}\times(0,T)}e^{-2s\tilde\beta_2}\left|\frac{\partial \theta}{\partial n}\right|^2\d\sigma\dt &\leq \iint_{\Sigma}e^{-2s\widehat{\beta_2}}\left|\frac{\partial \theta}{\partial n}\right|^2\d\sigma\dt \\ \label{est_local_theta}
&\leq C\int_{0}^{T}e^{-2s\widehat{\beta_2}}\|\theta(t)\|_{H^2(\Omega)}^2\dt.
\end{align}

Our task is now to estimate the weighted norm in the above expression. Multiplying the equation satisfied by $\theta$ by $e^{-2s\widehat{\beta_2}}\theta$ in $\Omega$ yields
\begin{equation*}\int_{\Omega}\left(\theta_t-\theta \Delta \theta\right) e^{-2s\widehat{\beta_2}}\theta \dx=\int_{\Omega}\left(-\frac{1}{\ell^2}\varphi\csin{\mathcal B_1}+\frac{1}{\gamma^2}\varphi \csin{\mathcal B_2}\right)e^{-s\widehat{\beta_2}}\theta\dx,\end{equation*}

Integrating by parts and since the weight $\widehat{\beta_2}$ only depends on time, we have
\begin{equation}
\label{aux_temp}
\begin{split}
\frac12\frac{d}{dt}\int_{\Omega}e^{-2s\widehat{\beta_2}}|\theta|^2\dx+&\int_{\Omega}e^{-2s\widehat{\beta_2}}|\nabla \theta|^2\dx\\
&=\int_{\Omega}\left(-\frac{1}{\ell^2}\varphi\csin{\mathcal B_1}+\frac{1}{\gamma^2}\varphi \csin{\mathcal B_2}\right)e^{-s\hat{\beta}_2}\theta\dx-s\int_{\Omega}e^{-2s\widehat{\beta_2}}(\widehat{\beta_2})_t|\theta|^2\dx.
\end{split}
\end{equation}
 Notice that the last term of \eqref{aux_temp} is nonnegative thanks to the nondecreasing nature of the weight $\widehat{\beta_2}$. From this remark and using H\"older and Young inequalities, we get
 \begin{equation}
\label{aux8}
\begin{split}
\frac12\frac{d}{dt}\int_{\Omega}e^{-2s\widehat{\beta}_2}|\theta|^2\dx+&\int_{\Omega}e^{-2s\widehat{\beta_2}}|\triangledown \theta|^2\dx\\
&\leq \frac{C}{\ell^4}\int_{\mathcal B_1}|\varphi|e^{-2s\widehat{\beta_2}}\dx+\frac{C}{\gamma^4}\int_{\mathcal B_2}|\varphi|e^{-2s\widehat{\beta_2}}\dx+\frac12\int_{\Omega}e^{-2s\widehat{\beta_2}}|\theta|^2\dx, 
\end{split}
\end{equation}
and applying Gronwall's lemma to \eqref{aux8} followed by integration by parts yields
 \begin{equation}
\label{aux9}
\iint_Q e^{-2s\widehat{\beta_2}}|\theta|^2\dx\dt\leq C\left(\frac{1}{\ell^4}\int_0^T\!\!\!\!\int_{\mathcal B_1}|\varphi|^2e^{-2s\widehat{\beta_2}}\dx\dt+\frac{1}{\gamma^4}\int_0^T\!\!\!\!\int_{\mathcal B_2}|\varphi|^2e^{-2s\widehat{\beta_2}}\dx\dt\right).
\end{equation}

A similar analysis can be developed to obtain the estimate
\begin{equation}
\label{aux10}
\iint_{Q} e^{-2s\widehat{\beta_2}}|\Delta\theta|^2\dx\dt\leq C\left(\frac{1}{\ell^4}\int_0^T\!\!\!\!\int_{\mathcal B_1}|\varphi|^2e^{-2s\widehat{\beta_2}}\dx\dt+\frac{1}{\gamma^4}\int_0^T\!\!\!\!\int_{\mathcal B_2}|\varphi|^2e^{-2s\widehat{\beta_2}}\dx\dt\right).
\end{equation}
Indeed, it is enough to multiply the second equation of \eqref{adjunto1} by $e^{-2s\widehat{\beta_2}}\Delta\theta$ in $\Omega$, integrate by parts and argue as above. Therefore, we deduce
\begin{align}\notag 
\int_{0}^Te^{-2s\widehat{\beta_2}}\|\theta(t)\|_{H^2(\Omega)}^2\dt&=\iint_{Q}e^{-2s\widehat{\beta_2}}\left(|\theta|^2+
|\Delta \theta|^2\right)\dx\dt \\ \label{est_theta_h2}
&\leq C\left(\frac{1}{\ell^4}\int_0^T\!\!\!\!\int_{\mathcal B_1}|\varphi|^2e^{-2s\widehat{\beta_2}}\dx\dt+\frac{1}{\gamma^4}\int_0^T\!\!\!\!\int_{\mathcal B_2}|\varphi|^2e^{-2s\widehat{\beta_2}}\dx\dt\right).
\end{align}

Putting together estimates \eqref{est_local_theta}, \eqref{est_theta_h2} and taking into account relation \eqref{rel_eta1_eta2} allow us to conclude that
\begin{equation}\label{est_local_fin}
\iint_{S^{\prime}\times(0,T)}e^{-2s\tilde\beta_2}\left|\frac{\partial \theta}{\partial n}\right|^2\dx\dt  \leq C\left(\frac{1}{\ell^4}\int_0^T\!\!\!\!\int_{\mathcal B_1}|\varphi|^2e^{-2s\tilde{\beta}_1}\dx\dt+\frac{1}{\gamma^4}\int_0^T\!\!\!\!\int_{\mathcal B_2}|\varphi|^2e^{-2s\tilde{\beta}_1}\dx\dt\right).
\end{equation}
Thus, thanks to hypothesis \eqref{loc_teo2} and the special selection of the weight functions \eqref{props_brazil}--\eqref{prop_impor}, we have estimated the local boundary term of $\theta$ in function of some localized terms depending on $\varphi$. 

To conclude, it is enough to combine estimates \eqref{aux7} and \eqref{est_local_fin} and take the parameters $\ell,\gamma$ large enough to absorb the remaining terms. Then, \eqref{obs_ineq_dif} follows from the resulting expression by setting $\varrho_2(t):=e^{s\tilde\beta_2^\star}$, where we have denoted $\tilde{\beta}_2^\star(t)=\max_{x\in\overline{\Omega}}\tilde\beta_2(x,t)$, and recalling that $\mathcal S^\prime\subset\subset \mathcal O$.  This ends the proof. 
\end{proof}


\section{All controls on the boundary}\label{sec_bound}

In this section, we shall discuss the hierarchic control problem for the system
\begin{equation}\label{sys_sec3}
\begin{cases}
w_t-\Delta w=\psi, & \quad \text{in Q}, \\
w=h\chi_{\cbd_1}+ v{\csbd}_2&\quad \text{on } \Sigma, \\
w(x,0)=w^0(x), &\quad \text{in } \Omega.
\end{cases}
\end{equation}
Notice that this time both controls are placed on the boundary and is a natural combination of the two previous problems. 

It is not difficult to prove that for any fixed $h\in L^2(\Sigma)$ the exists a unique saddle point $(\bar v,\bar \psi)\in L^2(\Sigma)\times L^2(Q)$ for the cost functional \eqref{func_rob}. Indeed, the procedure is practically the same as in Section \ref{ex_uniq_saddle} since the leader control $h$ is fixed and participates in a indirect way. 

As in Section \ref{sec_bound_leader}, once the saddle point has been characterized, the control $h$ of minimal norm can be obtained by solving a suitable minimization problem. This process would require to prove an observability of the form
\begin{equation}\label{obs_ineq_sec3}
\|\varphi(x,0)\|^2_{H_0^1(\Omega)}+\iint_{Q}\vartheta_3^{-2}(t)|\theta|^2\dx\dt\leq C\iint_{ \cbd_1\times(0,T)}\left|\frac{\partial\varphi}{\partial n}\right|^2\d\sigma\dt,
\end{equation}
 for the solutions to the adjoint system 
\begin{equation}\label{adj_sys_3}
\begin{cases}
-\varphi_t-\Delta \varphi= \theta\chi_{\mathcal O_d} \quad&\textnormal{in }Q, \\
\theta_t-\Delta \theta=\frac{1}{\gamma^2}\varphi \quad& \textnormal{in }Q, \\
\varphi=0, \quad \theta=\frac{1}{\ell^2}\frac{\partial \varphi}{\partial n}{\csbd}_2 \quad &\textnormal{on $\Sigma$,} \\
\varphi(x,T)=\varphi^T(x), \quad \theta(x,0)=0 \quad &\textnormal{in $\Omega$},
\end{cases}
\end{equation}
where $\vartheta_3(t)$ is an appropriate weight. Note that this system is the same as \eqref{adj_sys_foll}, but the observability we need now has a boundary observation instead of a distributed one. 

We could opt to apply the Carleman inequality, presented in Lemma \ref{lemma_car_boundary}, to the first equation of \eqref{adj_sys_3} while the Carleman estimate \eqref{car_boundary} to the second equation in \eqref{adj_sys_3}. Nevertheless, the definition of their respective weights are based on the selection of the functions $\eta_i$ (see eq. \eqref{constr_bound}) and $\eta^0$ (see eq. \eqref{constr_1}) and their different nature complicates the absorption of the lower order terms at the moment of adding both estimates. Moreover, as a consequence of applying \eqref{car_boundary} to the equation verified by $\theta$, we would have local term depending on $\theta$ for which it is not clear the procedure to eliminate it. 

As mentioned in Section \ref{sec_intro}, we shall present a hierarchic control result for the case when $\psi\equiv 0$ and the optimal control objective is modified as follows. 

Let us choose any function $\rho_\star\in C^\infty([0,T])$, such that $\rho_\star(t)\geq e^{\frac{s\bar\alpha}{2}}$ with $\alpha$ defined as
\begin{equation*}
\bar\alpha(x,t)=\frac{e^{2\lambda\|\bar\eta\|_{\infty}}-e^{\lambda\bar\eta(x)}}{t^2(T-t)^2}
\end{equation*}
where $\bar\eta$ is a function verifying properties \eqref{constr_bound}. 
%
%
Consider the optimization problem
\begin{equation*}
\min_{v\in L^2(\Sigma)} \mathcal I(v;h)
\end{equation*}
for the cost functional
\begin{equation*}
\mathcal I(v;h)=\frac{1}{2}\iint_{\mathcal O_d\times(0,T)}|w-w_d|^2\dx\dt+\frac{\ell^2}{2}\iint_{\Sigma}\rho_\star^2|v|^2\d\sigma\dt.
\end{equation*}
This is classical optimal control problem (cf. \cite{Lions_optim}) and the existence and uniqueness of its minimizer is  guaranteed if \eqref{func_teo3} is lower semicontinuous, strictly convex and coercive. The first two conditions can be readily verified while the coercivity follows from the fact that $e^{\frac{s\bar\alpha}{2}}\geq C>0$ for all $(x,t)\in Q$. 

The characterization of the minimum can be carried out as in the previous sections and leads to the optimality system
\begin{equation*}
\begin{cases}
w_t-\Delta w=0, &\quad  \text{in Q}, \\
-r_t-\Delta r= (w-w_d)\chi_{\mathcal O_d}  &\quad  \text{in Q},\\ 
w=h\csin{\Gamma_1}+ \frac{1}{\ell^2}\frac{\partial r}{\partial n}\rho_\star^{-2}{\csbd}_2, \quad r=0 &\quad \text{on } \Sigma, \\
w(x,0)=w^0(x),\quad r(x,T)=0 &\quad \text{in } \Omega.
\end{cases}
\end{equation*}

The next step is then to prove \eqref{obs_ineq_sec3} for the solutions to the adjoint system
\begin{equation}\label{adj_sys_3}
\begin{cases}
-\varphi_t-\Delta \varphi= \theta\chi_{\mathcal O_d} \quad&\textnormal{in }Q, \\
\theta_t-\Delta \theta=0 \quad& \textnormal{in }Q, \\
\varphi=0, \quad \theta=\frac{1}{\ell^2}\frac{\partial \varphi}{\partial n}\rho_{\star}^{-2}{\csbd}_2 \quad &\textnormal{on $\Sigma$,} \\
\varphi(x,T)=\varphi^T(x), \quad \theta(x,0)=0 \quad &\textnormal{in $\Omega$}.
\end{cases}
\end{equation}
The introduction of the weight $\rho_\star$ will help us to obtain a Carleman estimate for the solutions of \eqref{adj_sys_3} without the necessity of applying a Carleman inequality to the equation verified by $\theta$ and thus avoiding the problems mentioned before. The result is the following:
\begin{proposition}\label{prop_4_final}
Assume that $\ell$ is large enough. There exist a constant $\lambda_0$ such that for any $\lambda\geq \lambda_0$ and a constant $C>0$ such that the solution $(\varphi,\theta)$ to \eqref{adj_sys_3} satisfies 
\begin{equation}\label{car_final_sec3}
\iint_{Q}e^{-2s\bar\alpha}(s\bar\xi)^3|\varphi|^2\dx\dt+\iint_{Q}e^{-2s\bar\alpha^\star}|\theta|^2\dx\dt\leq C\iint_{\mathcal \mathcal O_1\times(0,T)}e^{-2s\bar\alpha}s\bar\xi\left|\frac{\partial \varphi}{\partial n}\right|^2\dx\dt
\end{equation}
for all $s$ large enough and every $\varphi^T\in H_0^1(\Omega)$. Here, we have used the notation
\begin{equation*}
\bar\xi(x,t)=\frac{e^{\lambda\bar\eta(x)}}{t^2(T-t)^2} \quad\text{and}\quad \bar\alpha^\star(t)=\frac{e^{2\lambda\|\eta\|_\infty}-e^{\lambda\min_{x\in \overline\Omega}\eta}}{t^2(T-t)^2}.
\end{equation*}
\end{proposition}

\begin{proof}
Let us set $\mathcal S^\prime\subset\subset \mathcal O_1$. Then, we apply inequality \eqref{Est1} (with the corresponding weights $\bar\alpha$, $\bar\xi$ and $m=0$) to the first equation of system \eqref{adj_sys_3}. By fixing $\lambda_0>0$, we obtain
\begin{equation}\label{car_init_sec3}
\iint_Qe^{-2s\bar\alpha}(s\bar\xi)^3|\varphi|^2\dx\dt\leq C\left(\iint_{Q}e^{-2s\bar\alpha}|\theta\chi_{\mathcal O_d}|^2\dx\dt+\iint_{\mathcal S^\prime\times(0,T)}e^{-2s\bar\alpha}s\bar\xi\left|\frac{\partial \varphi}{\partial n}\right|\dx\dt\right)
\end{equation}
valid for $s$ and $\lambda$ large enough. 

Now, we define $\widehat{\theta}=e^{-s\bar\alpha^\star}\theta$. Then, $\widehat{\theta}$ is solution to the system 
\begin{equation*}
\begin{cases}
-\widehat\theta_t-\Delta \widehat\theta=(e^{-s\alpha^\star})_t\theta &\quad \text{in } Q, \\
\widehat\theta=e^{-s\bar\alpha^\star}\frac{\partial\varphi}{\partial n}\rho_\star^{-2}{\csbd}_{2} &\quad \text{on } \Sigma, \\
\widehat{\theta}(\cdot,0)=0 &\quad\text{in } \Omega.
\end{cases}
\end{equation*}
Notice that $e^{-s\bar\alpha^\star}\in C^{\infty}([0,T])$, thus the following energy estimate holds
\begin{equation}\label{est_theta_gorro}
\|\widehat\theta\|_{W(Q)}\leq C\left(\|e^{-s\bar\alpha^\star}s(\bar\xi^\star)^{3/2}\theta\|_{L^2(Q)}+\left\|\frac{1}{\ell^2}e^{-s\bar\alpha^\star}\frac{\partial\varphi}{\partial n}\rho_\star^{-2}\right\|_{H^{\frac12,\frac14}(\Sigma)}\right).
\end{equation}
Here, we have denoted $\bar{\xi}^\star(t)=\frac{e^{\lambda\min_{x\in \overline\Omega}\bar\eta}}{t^2(T-t)^2}$ and used the fact that $|\bar\alpha^\star_t|\leq C({\bar\xi}^\star)^{3/2}$ . Arguing as in the proof of Proposition \ref{prof_regularity}, we deduce from \eqref{est_theta_gorro} that
\begin{equation}\label{est_bar_theta}
\iint_{Q}e^{-2s\bar\alpha^\star}|\theta|^2\dx\dt \leq C\iint_{Q}e^{-2s\bar\alpha^\star}s(\bar\xi^\star)^3|\theta|^2\dx\dt+\frac{C}{\ell^4}\|e^{-s\bar\alpha^\star}\rho_\star^{-2}\varphi\|^2_{H^{2,1}(Q)}.
\end{equation}

Adding estimates \eqref{car_init_sec3} and \eqref{est_bar_theta}, and taking into account that $\mathcal O_d\subset\Omega$, we get
\begin{align} \notag
\iint_{Q}&e^{-2s\bar\alpha}(s\bar\xi)^3|\varphi|^2\dx\dt+\iint_{Q}e^{-2s\bar\alpha^\star}|\theta|^2\dx\dt \\ \label{car_inter_sec3}
&\leq C\iint_{\Gamma_1\times(0,T)}e^{-2s\bar\alpha}s\bar\xi\left|\frac{\partial \varphi}{\partial n}\right|^2\d\sigma\dt+C\iint_{Q}e^{-2s\bar\alpha^\star}(s\bar\xi^\star)^3|\theta|^2\dx\dt+\frac{C}{\ell^4}\|e^{-s\bar\alpha^\star}\rho_\star^{-2}\varphi\|^2_{H^{2,1}(Q)}
\end{align}
for all $s$ large enough. 

At this point, we set $s$ to a fixed value sufficiently large and reasoning as in Proposition \ref{prop_obs_ineq_dif}, we will absorb the last two terms in the above expression by taking the parameter $\ell$ large enough. Since $\theta(\cdot,0)=0$ and
\begin{equation*}
e^{-2s\bar\alpha^\star}({\bar\xi^\star})^3\leq C, \quad \forall (x,t)\in Q,
\end{equation*}
we can use a classical (non-weighted) energy estimates for the solutions to $\theta$. More precisely, we can bound the second term in the right-hand side of \eqref{car_inter_sec3} as
\begin{align}\notag
\iint_{Q}e^{-2s\bar\alpha^\star}({\bar\xi^\star})^3|\theta|^2\dx\dt&\leq C\|\theta\|^2_{L^2(Q)}  \\ \label{est_nonwei}
&\leq \frac{C}{\ell^4}\|\rho_\star^{-2}\varphi\|^2_{H^{2,1}(Q)}. 
\end{align}

To conclude, it is enough to obtain an estimate for $\widehat{\varphi}:=e^{-s\bar\alpha^\star}\varphi$ in $H^{2,1}(Q)$. Indeed, since $\rho_\star\geq e^{\frac{s\alpha}{2}}$, we can use such estimate to bound the last term in \eqref{car_inter_sec3} and \eqref{est_nonwei}. For this, notice that $\widehat{\varphi}$ verifies
\begin{equation*}
\begin{cases}
-\widehat\varphi_t-\Delta \widehat\varphi=-(e^{-s\alpha^\star})_t\varphi+e^{-s\alpha^\star}\theta\chi_{\mathcal O_d} &\quad \text{in } Q, \\
\widehat\varphi=0 &\quad \text{on } \Sigma, \\
\widehat{\varphi}(\cdot,T)=0 &\quad\text{in } \Omega,
\end{cases}
\end{equation*}
then, we immediately deduce that
\begin{equation}\label{est_final}
\|\widehat{\varphi}\|_{H^{2,1}(Q)}\leq C\left(\|e^{-s\bar\alpha^\star}(\xi^\star)^{3/2}\varphi\|_{L^2(Q)}+\|e^{-s\bar\alpha^\star}\theta\|_{L^2(Q)}\right).
\end{equation}
Finally, using \eqref{est_final} to estimate in \eqref{car_inter_sec3} and \eqref{est_nonwei} the desired inequality \eqref{car_final_sec3} follows by taking $\ell>>1$. This concludes the proof.
\end{proof}

\begin{remark} The observability inequality \eqref{obs_ineq_sec3} is a direct consequence of \eqref{car_final_sec3} and can be found arguing as in the previous sections. For brevity, we will omit the proof. 
\end{remark}

\section{Concluding remarks}\label{sec_conclusion}

In this paper, we have considered the robust hierarchic control problem for the simple case of a heat equation. However, there are several open questions and related remarks that are worth mentioning.

\begin{enumerate}
\item \textit{Nonlinear problems.} So far, we have focused on studying linear control problems. A further step is to analyze the robust hierarchic control problem for systems of the form
\begin{equation}\label{nonlin_trans}
\begin{cases}
y_t-\Delta y+f(y)=h\chi_{\omega}+\psi &\quad \text{in } Q,\\
y=v\csbd &\quad\text{in }\Sigma, \\
y(x,0)=y^0(x) &\quad\text{in } \Omega,
\end{cases}
\end{equation}
where $f\in C^2(\mathbb{R})$ is a globally Lipschitz function. The well-posedness of \eqref{nonlin_trans} in the  functional space \eqref{space_trans} can be understood as in \cite[Section 8.2]{pighin} and therefore the robust control problem (see Definition \ref{defi_rob}) is meaningful. 

For proving the existence and uniqueness of the saddle point, we need to obtain the first and second order Frechet derivatives of the input-to-state operator $G:(v,\psi)\to y$ where $y$ is the solution to \eqref{nonlin_trans}, {as well as their regularity.} This can be done by following \cite{vhs_deT_rob}. Nevertheless, for proving an analog to Proposition \ref{verif_cond} for the solutions to \eqref{nonlin_trans} we need some additional embedding results (cf. \cite[Proof of Prop. 2]{vhs_deT_rob}) and, in this case, it is not so clear how to obtain them. Thus, it remains as an open problem.
\item \textit{On the hypothesis \eqref{loc_teo2}}. We have proved Theorem \ref{teo_main2} by establishing an observability inequality where hypothesis \eqref{loc_teo2} was heavily used. Indeed, the construction of two different Carleman weights fulfilling \eqref{props_brazil}--\eqref{prop_impor} and then the elimination of the second boundary observation in \eqref{aux7} rely on such hypothesis. Notice also that, as a consequence, we were obliged to consider a perturbation $\psi$ localized in the domain $\mathcal B_2\times(0,T)$ for the robust control part but, at that level, \eqref{loc_teo2} is not necessary to determine the existence and uniqueness of the saddle point. It is therefore interesting to prove Theorem \ref{teo_main2} without using hypothesis \eqref{loc_teo2} or by considering an alternative procedure that allows to have a perturbation $\psi$ in the whole domain $Q$. 

\item \textit{Remarks on the Stackelberg robust control with all controls localized in the boundary}. We devoted Section \ref{sec_bound} to prove a hierarchic control result for the heat equation in the case where the leader and the follower are placed on the boundary. This was possible due to the weighted functional \eqref{func_teo3} whose minimization enforces the control $v$ to vanish exponentially as $t$ goes to $0$ and $T$. 

We could have taken into account the effect of a perturbation affecting the system (see eq. \eqref{sys_sec3}) by considering a cost functional of the form
\begin{equation*}
\mathcal K(v,\psi;h)=\frac{1}{2}\iint_{\mathcal O_d\times(0,T)}|w-w_d|^2\dx\dt+\frac{\ell^2}{2}\iint_{\Sigma}\rho_\star^2|v|^2{\csbd}_2\d\sigma\dt-\frac{\gamma^4}{2}\iint_{Q}\rho_\star^2|\psi|\dx\dt.
\end{equation*}
For this functional, we can prove the existence of a saddle point if $\gamma$ is large enough. However, the introduction of the weight function indicates that only perturbations vanishing at $t=0$ and $t=T$ are allowed. From the practical point of view, it makes sense to consider a control $v$ with these properties, since is at the choice of the designer, however, the perturbations are not at hand and restricting its behavior to such class of functions is not accurate.

\item \textit{A Stackelberg-Nash controllability with all controls on the boundary}. Theorem \ref{teo3} can be extended to the case when more followers participate in the definition of the optimal control problem. More precisely, let us consider the system
\begin{equation}\label{sys_sec3}
\begin{cases}
w_t-\Delta w=0, & \quad \text{in Q}, \\
w=h\chi_{\cbd}+ v^1\rho_{\Gamma_1}+v^2\rho_{\Gamma_2}&\quad \text{on } \Sigma, \\
w(x,0)=w^0(x), &\quad \text{in } \Omega,
\end{cases}
\end{equation}
where $\Gamma,\Gamma_i\subset \partial \Omega$ are open sets with $\Gamma\cap\Gamma_i=\emptyset$, $i=1,2$. Also, consider the functionals
\begin{equation}\label{func_nash}
\mathcal I_i(v^1,v^2;h)=\frac{1}{2}\iint_{\mathcal O_{i,d}\times(0,T)}|w-w_{i,d}|^2\dx\dt+\frac{\ell_i^2}{2}\iint_{\Sigma}\rho_\star^2|v^{i}|^2\d\sigma\dt, \quad i=1,2.
\end{equation}
where $w_{i,d}$ in $L^2(Q)$ are given functions and $\mathcal{O}_{i,d}\subset \Omega$ are arbitrary observation sets.

Here, the goal is to design $v^1$ and $v^2$ so that a \textit{Nash equilibrium} for the functionals \eqref{func_nash} is attained, this is,  for any fixed $h$, we look for a pair $(\bar v^1,\bar v^2)$ verifying  
\begin{equation}\label{nash}
I_1(\bar v^1,\bar v^2;h)=\min_{v^1\in L^2(\Sigma)}I_1(v^1,\bar v^2), \quad I_2(\bar v^1,\bar v^2;h)=\min_{v^2\in L^2(\Sigma)}I_1(\bar v^1, v^2).
\end{equation}
Adapting the results of \cite{araruna,vhs_corri}, it can be seen that, choosing $\ell_i$ large enough, there exists a unique pair $(\bar v^1,\bar v^2)$ satisfying \eqref{nash}. The characterization of the Nash equilibrium then leads to the optimality system
\begin{equation*}
\begin{cases}
w_t-\Delta w=0, &\quad  \text{in Q}, \\
-r^{i}_t-\Delta r^{i}= (w-w_{i,d})\chi_{\mathcal O_d}  &\quad  \text{in Q},\\ 
w=h\chi_{\cbd}+ \frac{1}{\ell_1^2}\frac{\partial r^{1}}{\partial n}\rho_\star^{-2}\rho_{\Gamma_1}+\frac{1}{\ell_2^2}\frac{\partial r^{2}}{\partial n}\rho_\star^{-2}\rho_{ \Gamma_2}, \quad r^{i}=0 &\quad \text{on } \Sigma, \\
w(x,0)=w^0(x),\quad r^{i}(x,T)=0 &\quad \text{in } \Omega, \quad i=1,2,
\end{cases}
\end{equation*}
and the null controllability for this system can be addressed by obtaining a suitable observability inequality for the solutions to the adjoint system
\begin{equation}\label{adj_conclusion}
\begin{cases}
-\varphi_t-\Delta \varphi=\theta^1\chi_{\mathcal O_{2,d}}+\theta^2\chi_{\mathcal O_{2,d}}, &\quad  \text{in Q}, \\
\theta^{i}_t-\Delta \theta^{i}= 0 &\quad  \text{in Q},\\ 
\varphi=0, \quad \theta^{i}=\frac{1}{\ell_i^2}\frac{\partial \varphi}{\partial n}\rho_\star^{-2}\chi_{\Gamma_i}, &\quad \text{on } \Sigma, \\
\varphi(x,T)=\varphi^T(x),\quad \theta^{i}(x,0)=0 &\quad \text{in } \Omega, \quad i=1,2.
\end{cases}
\end{equation}
Thanks to the weight $\rho_\star$, this inequality can be obtained as in the proof of Proposition \ref{prop_4_final}: it is enough to apply the Carleman inequality \eqref{Est1} to the first equation of \eqref{adj_conclusion} and then only use energy estimates to absorb the remaining terms depending on $\theta^{i}$ for $i=1,2$. For this same reason, unlike \cite{araruna,vhs_corri}, there is not need to impose additional hypotheses on the sets $\mathcal O_{i,d}$.

\item \textit{Changing the objective of the leader}. As far as the authors' knowledge, all of the papers devoted to hierarchic control consider a controllability (either null or approximate) constraint as the leader's objective. An interesting problem that arises is to study the case where the leader is now in charge of an \emph{insensitizing} problem (see, e.g., \cite{deteresa2000}). To fix ideas, consider the system 
\begin{equation*}
\begin{cases}
y_t-\Delta y=\xi + h\chi_{\omega}+v\chi_{\mathcal O} &\quad\text{in }Q, \\
y=0 &\quad\text{in } \Sigma, \\
y(x,0)=y_0(x)+\tau\bar y  &\quad\text{in } \Omega,
\end{cases}
\end{equation*}
where $\xi\in L^2(Q)$ is a given source term. The data of the system is incomplete in the following sense: $\bar y\in L^2(\Omega)$ is unknown with $\|y\|_{L^2(\Omega)}=1$ and $\tau\in\mathbb R$ is unknown and small enough. 

As usual, the follower $v$ will be in charge of a classical optimal control problem (i.e., minimize a functional like \eqref{func_rob} with $\gamma\equiv 0$) and for $h$ and $\xi$, the expected optimality system takes the form
\begin{equation}\label{sys_insi}
\begin{cases}
y_t-\Delta y=\xi + h\chi_{\omega}-\frac{1}{\ell^2}p\chi_{\mathcal O} &\quad\text{in }Q, \\ 
-p_t-\Delta p=(y-y_d)\chi_{\mathcal O_d} &\quad\text{in }Q \\
y=0, \quad p=0 &\quad\text{in } \Sigma, \\
y(x,0)=y_0(x)+\tau\bar y, \quad p(x,T)=0  &\quad\text{in } \Omega.
\end{cases}
\end{equation}
Now, consider a differentiable functional $\Psi$ defined on the sets of solutions to \eqref{sys_insi}, for instance, for some observation set $\mathcal S\subset \Omega$ we define
\begin{equation*}
\Psi(y):= \frac{1}{2}\iint_{\mathcal S\times(0,T)}|y|^2\dx\dt.
\end{equation*}
We say that a control $h$ insensitizes $\Psi(y)$ for the initial datum $y_0$ and the source term $\xi$ if
\begin{equation}\label{obj_ins}
\left.\frac{\partial\Psi}{\partial\tau}\right|_{\tau=0}=0, \quad \forall \bar y\in L^2(\Omega). 
\end{equation}
As common in other insensitizing control problems, \eqref{obj_ins} is equivalent to a non-standard controllability problem. For our case, this translates into finding $h$ such that $z(x,0)=0$ where $(y,p,z,q)$ is the solution to
\begin{equation}\label{sys_insi}
\begin{cases}
y_t-\Delta y=\xi + h\chi_{\omega}-\frac{1}{\ell^2}p\chi_{\mathcal O} &\quad\text{in }Q, \\ 
-p_t-\Delta p=(y-y_d)\chi_{\mathcal O_d} &\quad\text{in }Q \\
-z_t-\Delta z= q+y\chi_\mathcal{S} &\quad\text{in }Q \\
q_t-\Delta q=-\frac{1}{\ell^2}z\chi_{\mathcal O} &\quad\text{in }Q \\
y= p=z=q=0 &\quad\text{in } \Sigma, \\
y(x,0)=y_0(x), \quad p(x,T)=z(x,T)=q(x,0)=0  &\quad\text{in } \Omega.
\end{cases}
\end{equation}

This means that we have to control one component of a system of four coupled equations. As pointed out in \cite{vhs_honor}, the hierarchic controllability of coupled systems as \eqref{sys_insi} with one control force is not yet fully understood and further investigation is desirable. 

\renewcommand{\abstractname}{Acknowledgements}
\begin{abstract}
\end{abstract}
\vspace{-0.5cm}
The work of the first author was partially supported by project   IN102116 of DGAPA,
UNAM (Mexico). Also, the authors wish to acknowledge  Prof. Luz de Teresa (UNAM, Mexico) and Prof. Franck Boyer (IMT, France) for the fruitful discussions. 
\end{enumerate}  


\end{document}